\documentclass{amsart}
\usepackage[english]{babel}
\usepackage{amsmath, amsthm, amssymb}
\usepackage{enumerate}
\usepackage{geometry}

\usepackage{hyperref}
\usepackage{cleveref}
\usepackage{comment}
\theoremstyle{plain}
\newtheorem{theorem}{Theorem}[section]
\newtheorem{proposition}[theorem]{Proposition}
\newtheorem{lemma}[theorem]{Lemma}
\newtheorem{corollary}[theorem]{Corollary}
\theoremstyle{definition}
\newtheorem{definition}[theorem]{Definition}

\newcommand{\C}{\mathbb{C}}
\newcommand{\R}{\mathbb{R}}
\newcommand{\Q}{\mathbb{Q}}
\newcommand{\Z}{\mathbb{Z}}
\renewcommand{\O}{\mathcal{O}}

\DeclareMathOperator{\Tr}{Tr}
\DeclareMathOperator{\Nm}{N}
\DeclareMathOperator{\SL}{SL}
\DeclareMathOperator{\rank}{rank}
\DeclareMathOperator{\sgn}{sgn}
\DeclareMathOperator{\NF}{NF}

\title{Representing rational integers by generalized quadratic forms over quadratic fields}
\author[]{Ond\v{r}ej Chwiedziuk}
\author[]{Mat\v{e}j Dole\v{z}\'{a}lek}
\author[]{Emma P\v{e}chou\v{c}kov\'{a}}
\author[]{Zden\v{e}k Pezlar}
\author[]{Om Prakash}
\author[]{Giuliano Romeo}
\author[]{Anna R\r{u}\v{z}i\v{c}kov\'{a}}
\author[]{Mikul\'{a}\v{s} Zindulka}
\address{Charles University, Faculty of Mathematics and Physics, Department of Algebra,
Sokolovsk\'{a} 83, 186 75 Praha 8, Czech Republic}
\email{ondrachwiedziuk@gmail.com}
\email{matej@gimli.ms.mff.cuni.cz}
\email{emma.pechouckova@gmail.com}
\email{zdendapezlar@seznam.cz}
\email{omprakash@ksom.res.in}
\email{anna.ruzickova.13@gmail.com}
\email{mikulas.zindulka@matfyz.cuni.cz}
\address{Department of Mathematical Sciences ``Giuseppe Luigi Lagrange", Politecnico di Torino, Turin, Italy}
\email{giuliano.romeo@polito.it}
\thanks{We acknowledge support by Czech Science Foundation (GA\v{C}R) grant 26-20514S, and Charles University programmes PRIMUS/24/SCI/010 and UNCE/24/SCI/022, and GAUK projects 134824 and 236225.}

\begin{document}

\begin{abstract}
We investigate generalized quadratic forms with values in the set of rational integers over quadratic fields. We characterize the real quadratic fields which admit a positive definite binary generalized form of this type representing every positive integer. We also show that there are only finitely many such fields where a ternary generalized form with these properties exists.
\end{abstract}

\maketitle

\section{Introduction}
\label{secIntro}

A positive definite quadratic form over $ \Z $ is called \emph{universal} if it represents every positive integer $ n $. The history of universal quadratic forms begins with Lagrange's proof of the Four Square Theorem in 1770: \textit{Every positive integer $ n $ is of the form $ x^2+y^2+z^2+w^2 $}. Diagonal quaternary universal forms were classified by Ramanujan \cite{Ra}, who stated that there exist $ 55 $ such forms. Dickson \cite{Di} corrected Ramanujan's statement by pointing out that the form $ x^2+2y^2+5z^2+5u^2 $ does not represent $ 15 $. Thus there exist exactly $ 54 $ diagonal universal forms. Dickson \cite{Di2} also proved various results on representation by non-diagonal quaternary forms.

In 1993, John H. Conway and W. A. Schneeberger proved the celebrated 15-Theorem. A quadratic form having an integer matrix is called \emph{classical}.

\begin{theorem}[15-Theorem, Conway-Schneeberger]
\label{thmFifteen}
A classical quadratic form $ Q $ over $ \Z $ is universal if and only if it represents the integers
\[
	1, 2, 3, 5, 6, 7, 10, 14, \text{and } 15.
\]
\end{theorem}

The proof was simplified by Bhargava \cite{Bh}. As a corollary, one obtains all universal classical quaternary forms -- there are exactly $207$ of them. Bhargava and Hanke \cite{BH} extended the result to non-classical forms by proving the 290-Theorem: \textit{If a positive definite quadratic form $ Q $ over $ \Z $ represents every positive integer $ n \leq 290 $, then $ Q $ is universal.} In fact, for $ Q $ to be universal, it is enough that it represents every integer from a list of $ 29 $ numbers, the largest number on the list being $ 290 $. A corollary is that there are exactly $ 6436 $ universal quaternary forms \cite[Theorem~4]{BH}.

It is natural to extend the notion of universality from $ \Z $ to the ring of integers $ \O_K $ in a number field. Assume that the number field $ K $ is \emph{totally real}. A positive integral quadratic form over $ K $ is called \emph{universal} if it represents all totally positive integers in $ K $. A version of the local-global principal over number fields proved by Hsia, Kitaoka, and Kneser \cite{HKK} implies that a universal form over $ K $ always exists. Maa{\ss} \cite{Ma} showed that the sum of three squares $ x^2+y^2+z^2 $ is universal over $ \Q(\sqrt{5}) $. By a result of Siegel \cite{Si}, if a sum of squares is universal over $ K $, then $ K = \Q $ or $ \Q(\sqrt{5}) $.  In~1993, Chan, Kim, and Raghavan \cite{CKR} proved that a ternary universal quadratic form exists over a real quadratic field $ K = \Q(\sqrt{D}) $ if and only if $ D $ equals $ 2 $, $ 3 $, or $ 5 $. Moreover, they identified all ternary universal forms over these fields up to equivalence. For forms of higher ranks, B. M. Kim \cite{Ki} proved that there exist only finitely many $ D $'s such that $ \Q(\sqrt{D}) $ admits a universal form in $7$ variables (see also \cite{KKP2} for an explicit estimate on the size of $ D $). On the other hand, a $ 8 $-ary universal form exists over each $ \Q(\sqrt{n^2-1}) $ whenever $ n^2-1 $ is squarefree \cite{Ki2}. Blomer and Kala \cite{BK} showed: \textit{For any positive integer $ r $, there are infinitely many quadratic fields $ \Q(\sqrt{D}) $ that do \emph{not} have a universal lattice of rank $ \leq r $.} Lower bounds for the rank of a universal form depending on $ D $ were obtained by Kala and Tinková \cite{KT}.

One of the principal open problems in this area is Kitaoka's conjecture: \textit{There are only finitely many totally real number fields $ K $ having a ternary universal form.} The conjecture was proved by Kala and Yatsyna \cite{KY} for fields of an arbitrary but fixed degree $ d $.

For a more comprehensive account of the theory of universal quadratic forms, especially over number fields, we refer the reader to the survey article \cite{Ka2}.

In the same vein, one can study Hermitian forms over imaginary quadratic fields. We call a positive definite Hermitian form over an imaginary quadratic field \emph{$\Z$-universal} if it represents every $ n \in \Z_{\geq 1} $. Earnest and Khosravani \cite{EK} proved that universal binary Hermitian forms exist over $ K =\Q(\sqrt{m}) $ for only finitely many $ m $, and identified all such forms in the case when $ K $ has class number $1$. B. M. Kim, J. Y. Kim, and P.-S. Park \cite{KKP} proved an analogy of the 15-Theorem for Hermitian quadratic forms: \textit{If a positive definite Hermitian form $ H $ represents $1$, $2$, $3$, $5$, $6$, $7$, $10$, $13$, $14$, and $15$, then $ H $ is universal.} They also completely determined the minimal rank of a positive Hermitian form over all imaginary quadratic fields.

Our paper concerns generalized quadratic forms introduced recently by Browning, Pierce, and Schindler \cite[Definition~1.1]{BPS}. They can be defined for an arbitrary Galois extension $ K $ of $ \Q $ but we will consider them only over quadratic fields. A quadratic field $ K = \Q(\sqrt{D}) $ has two $ \Q $-automorphisms: the identity and the conjugation map
\begin{align*}
	\tau: K &\hookrightarrow K\\
    a+b\sqrt{D} &\mapsto a-b\sqrt{D}.
\end{align*}
Informally speaking, a generalized quadratic form may contain also the conjugates of the variables. For example,
\begin{equation}
\label{eqGex}
	G(z, w) = z^2-z\tau(z)+\tau(z)^2+w^2+w\tau(w)+\tau(w)^2
\end{equation}
is a generalized quadratic form (and we will show later that it represents every $ a \in \Z_{\geq 1} $ over $ \Q(\sqrt{2}) $). Thus, the concept generalizes both quadratic and Hermitian forms.

A generalized quadratic form is a special case of a generalized polynomial (see our \Cref{defGenPol}). The values of a generalized polynomial $ g $ in $ n $ variables with coefficients in $ K $ also lie in $ K $. However, we restrict our attention to generalized polynomials with values in $ \Q $ and find a minimal set of generators for the ring of such polynomials in~\Cref{thmBasis}. 

If $ G(z_1, z_2, \dots, z_n) $ is a generalized form in $ n $ variables and
\[
	G(\alpha_1, \alpha_2, \dots, \alpha_n) \in \Z,\quad \forall \alpha_1, \alpha_2, \dots, \alpha_n \in \O_K,
\]
then we say that $ G $ is \emph{$ \Z $-valued}. The properties of being positive definite, integral, and classical also readily extend to this setting. A positive definite integral $ \Z $-valued generalized form which represents every $ a \in \Z_{\geq 1} $ will be called \emph{$ \Z $-universal}. Our main theorem is the following.

\begin{theorem}
\label{thmBinary}
Let $ K = \Q(\sqrt{D}) $ where $ D \in \Z_{\geq 2} $ is squarefree.
\begin{enumerate}[i)]
\item Assume that $ D \equiv 2, 3 \pmod{4} $. A binary $ \Z $-universal generalized quadratic form exists over $ K $ if and only if $ D \in \{2,3,6,7,10\} $.
\item Assume that $ D \equiv 1 \pmod{4} $. A \emph{classical} binary $ \Z $-universal generalized quadratic form exists over $ K $ if and only if $ D = 5 $.
\end{enumerate}
\end{theorem}

We further extend this result to ternary generalized forms by proving the following.

\begin{theorem}
\label{thmTernary}
Let $ K = \Q(\sqrt{D}) $ where $ D \in \Z_{\geq 2} $ is squarefree. Assume that one of the following two conditions is satisfied:
\begin{enumerate}[i)]
\item $ D \equiv 2, 3 \pmod{4} $ and there exists a ternary $ \Z $-universal generalized quadratic form over $ K $,
\item $ D \equiv 1 \pmod{4} $ and there exists a \emph{classical} ternary $ \Z $-universal generalized quadratic form over $ K $.
\end{enumerate}
Then $ D \leq 110 $.
\end{theorem}

The proofs heavily rely on the 15-Theorem. A natural question is what happens when we remove the assumption that the form is classical in the case $ D \equiv 1 \pmod{4} $. By an application of the 290-Theorem, we can find the answer.

\begin{theorem}
\label{thmBinTer}
There are only finitely many real quadratic fields which admit a binary or ternary $ \Z $-universal generalized quadratic form.
\end{theorem}

We will prove these as \Cref{thmBinary2,thmTernary2,thmBinTer2} in \Cref{secBinary}. In contrast with the situation for binary and ternary forms, it is not very difficult to construct a generalized form in four variables which is $ \Z $-universal over every real quadratic field (and we do this later in \Cref{propFourVars}). This provides a fairly complete picture of the minimal rank required for a generalized form to be $ \Z $-universal.

The preceding discussion relates only to forms which are positive definite. As regards indefinite forms, we prove the following in \Cref{thmIndef}: If $ K = \Q(\sqrt{D}) $ where $ D \in \Z_{\geq 2} $ is squarefree, then the binary generalized form
\[
	G(z, w) = z\tau(z)+w\tau(w)
\]
represents every $ a \in \Z $.

If we start with an integral $ \Z $-valued generalized form and express the variables in the integral basis, then we end up with an integral quadratic form over $ \Z $ in twice as many variables (this will be proved formally later). For example, if we let $ z = x_1+y_1\sqrt{2} $ and $ w = x_2+y_2\sqrt{2} $ in~\eqref{eqGex}, then we obtain the quadratic form
\[
	Q(x_1, y_1, x_2, y_2) = x_1^2+6y_1^2+3x_2^2+2y_2^2.
\]
(It can be checked that $ Q $ is universal by the 15-Theorem.) We call $ Q $ the quadratic form \emph{associated} to $ G $. A natural question is which universal quaternary quadratic forms arise in this way. In other words, we ask which quaternary quadratic forms are associated to some $ \Z $-universal binary generalized form. We find all such forms (up to equivalence) for $ D \in \{2, 3 ,6, 7, 10\} $. We also find all quadratic forms associated to \emph{classical} $ \Z $-universal binary generalized forms (again, up to equivalence) for $ D = 5 $.

The rest of the paper is organized as follows. In \Cref{secGenPols}, we find a minimal set of generators of the ring of $ \Q $-valued generalized polynomials. The theory of generalized quadratic forms is developed in \Cref{secGenForms}. We prove~\Cref{thmBinary,thmTernary,thmBinTer} at the end of \Cref{secBinary} and \Cref{thmIndef} in \Cref{secIndef}. In the appendix,
we give the proofs of Lemmata \ref{lemmaExD3}-\ref{lemmaExD10} in Section \ref{SectionFillingInPRoofs55-59} and a more detailed approach to the proof of Theorem \ref{thmIndef} in Section \ref{secIndef2}. In Section \ref{secExamples2} we explicitly show all quaternary forms associated to $ \Z $-universal binary generalized forms for $ D \in \{2, 3, 6, 7, 10\} $ (and to classical ones for $ D = 5 $).

\section*{Acknowledgments}
This research project was accomplished at the Student Number Theory Seminar at Charles University. We thank Nicolas Daans, V\'{i}t\v{e}zslav Kala and Pavlo Yatsyna for their helpful suggestions and the anonymous referee for their constructive feedback.

\section{Generalized polynomials with rational values}
\label{secGenPols}
Throughout the article let $ K = \Q(\sqrt{D}) $ where $ D \in \Z\setminus\{0,1\} $ is squarefree. If we let
\[
    \omega_D = \begin{cases}
        \sqrt{D},&\text{if }D \equiv 2,3\pmod{4},\\
        \frac{1+\sqrt{D}}{2},&\text{if }D \equiv 1\pmod{4},
    \end{cases}
\]
then $ (1, \omega_D) $ is a basis of $ \O_K $.  For $ \alpha \in K $, we let $ \Tr(\alpha) = \alpha+\tau(\alpha) $ and $ \Nm(\alpha) = \alpha\tau(\alpha) $ be the \emph{trace} and \emph{norm} of $ \alpha $, respectively.

We begin by defining generalized polynomials over a quadratic field.

\begin{definition}
\label{defGenPol}
Let $ K = \Q(\sqrt{D}) $ where $ D \in \Z\setminus\{-1,0\} $ is squarefree. A \emph{generalized polynomial} in $ n $ variables over $ K $ is an expression of the form
\begin{equation}
\label{eqGenPol}
    g(z_1, z_2, \dots, z_n) = \sum_{\substack{i_1, j_1, \dots, i_n, j_n \geq 0\\i_1+j_1+\cdots+i_n+j_n \leq k}} \alpha_{i_1, j_1, \dots i_n, j_n}z_1^{i_1}\tau(z_1)^{j_1}\cdots z_n^{i_n}\tau(z_n)^{j_n}
\end{equation}
where $ k \in \Z_{\geq 0} $, $ \alpha_{i_1, j_1, \dots, i_n, j_n} \in K $ for every $ i_1, j_1, \dots, i_n, j_n \in \{0, 1, \dots, k\} $ and $\tau: K \hookrightarrow \C$ where $\tau(a+b\sqrt{D}) = a-b\sqrt{D}$ is the conjugation map.
\end{definition}

 We will often use the multiindex notation, where $ \mathbf{i} = (i_1, i_2, \dots, i_n) \in \Z_{\geq 0}^n $ is a multiindex,  $ |\mathbf{i}| = \sum_{r=1}^n i_r $ is the sum of its components, and for $ \mathbf{x} = (x_1, x_2, \dots, x_n) $ we let $ \mathbf{x}^\mathbf{i} = x_1^{i_1}x_2^{i_2}\cdots x_n^{i_n} $. 
 The lexicographic order on $ \Z_{\geq 0}^n $ is defined as follows: Let $ \mathbf{i}, \mathbf{j} \in \Z_{\geq 0}^n $, $ \mathbf{i} = (i_1, i_2, \dots, i_n) $ and $ \mathbf{j} = (j_1, j_2, \dots, j_n)$. Then $ \mathbf{i}<\mathbf{j} $ if and only if there exists $ s \in \{1, 2, \dots, n\} $ such that $ i_r = j_r $ for $ 1 \leq r < s $ and $ i_s < j_s $.

Hence if $ \mathbf{z} = (z_1, z_2, \dots, z_n) $, then we let $ \tau(\mathbf{z}) = (\tau(z_1), \tau(z_2), \dots, \tau(z_n)) $ and we can write the generalized polynomial $ g $ as
\[
    g(\mathbf{z}) = \sum_{|\mathbf{i}|+|\mathbf{j}|\leq k}\alpha_{\mathbf{i}\mathbf{j}}\mathbf{z}^\mathbf{i}\tau(\mathbf{z})^{\mathbf{j}}.
\]

If $ g $ is a generalized polynomial of the form~\eqref{eqGenPol}, then we define the \emph{degree} of $ g $ as
\[
    \deg(g) = \max\{|\mathbf{i}|+|\mathbf{j}|:\; \alpha_{\mathbf{i}\mathbf{j}}\neq 0\}.
\]

We see that $ g $ is a generalized polynomial in $ n $ variables if and only if there exists a polynomial $ \tilde{g} $ in $ 2n $ variables such that
\[
    g(z_1, z_2, \dots, z_n) = \tilde{g}(z_1, \tau(z_1), \dots, z_n, \tau(z_n)).
\]

We are particularly interested in generalized polynomials $ g $ with values in $ \Q $. They are characterized by the property $ g(\mathbf{a}) = \tau(g(\mathbf{a})) $ for every $ \mathbf{a} \in K^n $. We prove below that the coefficients of the two conjugate terms $ \mathbf{z}^{\mathbf{i}}\tau(\mathbf{z})^{\mathbf{j}} $ and $ \tau(\mathbf{z})^{\mathbf{j}}\mathbf{z}^{\mathbf{i}} $ in such a polynomial must be conjugate, i.e., $ \alpha_{\mathbf{j}\mathbf{i}} = \tau(\alpha_{\mathbf{i}\mathbf{j}}) $.

\begin{lemma}
\label{lemmaGenPolZero}
Let $ K = \Q(\sqrt{D}) $ where $ D \in \Z\setminus\{0, 1\} $ is squarefree and let $ g $ be a generalized polynomial in $ n $ variables over $ K $. If $ g(\mathbf{a}) = 0 $ for every $ \mathbf{a} \in K^n $, then $ g $ is the zero polynomial.
\end{lemma}
\begin{proof}
Let $ \tilde{g} $ be a polynomial in $ 2n $ variables such that
\[
    g(z_1, z_2, \dots, z_n) = \tilde{g}(z_1,\tau(z_1),\dots, z_n,\tau(z_n)).
\]

First, suppose that $ K $ is real. We know that $ \sigma(K) = \{(\alpha, \tau(\alpha)):\; \alpha\in K\} $ is dense in $ V = \R^2 $. Consequently, $ \sigma(K)^n $ is dense in $ \R^{2n} $. Because the polynomial $ \tilde{g} $ vanishes on $ \sigma(K)^n $, it vanishes on the whole set $ \R^{2n} $, and hence $ \tilde{g} $ is the zero polynomial.

Secondly, suppose that $ K $ is imaginary. In this case, $ \sigma(K) = K $ is dense in $ V = \C $. Thus $ K^n $ is dense in $ \C^n $ and the polynomial $ \tilde{g} $ is zero on the set
\[
    \{(z_1, \overline{z_1}, \dots, z_n, \overline{z_n}):\; z_1, \dots, z_n \in K\}.
\]

Let $ h $ be another polynomial in $ 2n $ variables defined by
\[
    h(x_1, y_1, \dots, x_n, y_n) = \tilde{g}(x_1+iy_1, x_1-iy_1, \dots, x_n+iy_n, x_n-iy_n).
\]
Since $ h $ vanishes on $ \R^{2n} $ and
\[
    g(z_1, z_2, \dots, z_n) = h\left(\frac{z_1+\overline{z_1}}{2},\frac{z_1-\overline{z_1}}{2i}, \dots, \frac{z_n+\overline{z_n}}{2}, \frac{z_n-\overline{z_n}}{2i}\right),
\]
$ g $ is the zero polynomial.
\end{proof}

\begin{proposition}
\label{propGenPol}
Let $ K = \Q(\sqrt{D}) $ where $ D \in \Z\setminus\{0,1\} $ is squarefree and let $ g $ be a generalized polynomial in $ n $ variables over $ K $ given by~\eqref{eqGenPol}. We have $ g(\mathbf{a}) \in \Q $ for every $ \mathbf{a} \in K^n $ if and only if $ \alpha_{\mathbf{j}\mathbf{i}} = \tau(\alpha_{\mathbf{i}\mathbf{j}}) $ for every $ \mathbf{i}, \mathbf{j} \in \Z_{\geq 0}^n $.
\end{proposition}
\begin{proof}
Consider the two generalized polynomials
\begin{align*}
    g(\mathbf{z})& = \sum_{|\mathbf{i}|+|\mathbf{j}|\leq k} \alpha_{\mathbf{i}\mathbf{j}}\mathbf{z}^\mathbf{i}\tau(\mathbf{z})^\mathbf{j} = \sum_{|\mathbf{i}|+|\mathbf{j}|\leq k} \alpha_{\mathbf{j}\mathbf{i}}\mathbf{z}^\mathbf{j}\tau(\mathbf{z})^\mathbf{i}
\end{align*}
and
\[
    \tau(g)(\mathbf{z}) = \sum_{|\mathbf{i}|+|\mathbf{j}|\leq k} \tau(\alpha_{\mathbf{i}\mathbf{j}})\tau(\mathbf{z})^\mathbf{i}\mathbf{z}^\mathbf{j}.
\]

The condition $ g(\mathbf{a}) \in \Q $ for every $ \mathbf{a} \in K $ is equivalent to the condition $ g(\mathbf{a})=\tau(g)(\mathbf{a}) $ for every $ \mathbf{a} \in K $. By~\Cref{lemmaGenPolZero}, this is equivalent to $ g-\tau(g) $ being the zero polynomial, which is in turn equivalent to $ \alpha_{\mathbf{j}\mathbf{i}} = \tau(\alpha_{\mathbf{i}\mathbf{j}}) $ for every $ \mathbf{i}, \mathbf{j} \in \Z_{\geq 0}^n $.
\end{proof}

Let $ K $ be a quadratic field and let $ \Gamma_K^n $ be the ring of $\Q$-valued generalized polynomials in $ n $ variables, i.e.,
\[
    \Gamma_K^n = \left\{g(z_1, \dots, z_n) \in K[z_1, \tau(z_1), \dots, z_n, \tau(z_n)]:\; g(\mathbf{a})\in \Q\text{ for every }\mathbf{a}\in K^n\right\}.
\]

For $ d \in \Z_{\geq 0} $, let
\[
    \Gamma_K^{n,d} = \left\{g \in \Gamma_K^n:\; \deg(g)\leq d\right\}.
\]
Clearly, $ \Gamma_K^{n,d} $ is a vector space over $\Q$.

The main theorem of this section is the following one.

\begin{theorem}
\label{thmBasis}
Let $ K = \Q(\sqrt{D}) $ where $ D \in \Z\setminus\{0,1\} $ is squarefree. If $ n \in \Z_{\geq 1} $, then
\[
    \Gamma_K^n = \Q\left[z_1+\tau(z_1),\sqrt{D}(z_1-\tau(z_1)), \dots, z_n+\tau(z_n), \sqrt{D}(z_n-\tau(z_n))\right].
\]
\end{theorem}

The first step is to prove the theorem for $ n = 1 $. In this case, we denote the single variable by $ z $ instead of $ z_1 $, thus
\[
    \Gamma_K^1 = \{g(z)\in K[z, \tau(z)]:\; g(\alpha)\in \Q\text{ for every }\alpha \in K\}.
\]

\begin{lemma}
\label{lemmaBasis1}
If $ K = \Q(\sqrt{D}) $ where $ D \in \Z\setminus\{0,1\} $ is squarefree, then $ \Gamma_K^1 = \Q\left[z+\tau(z), \sqrt{D}(z-\tau(z))\right] $.
\end{lemma}
\begin{proof}
Let $ S_1 = \Q\left[z+\tau(z), \sqrt{D}(z-\tau(z))\right] $. We have $ S_1 \subset \Gamma_K^1 $ because $ z+\tau(z) $ and $ \sqrt{D}(z-\tau(z)) $ are $ \Q $-valued, and we want to prove $ \Gamma_K^1 \subset S_1 $.

We claim that $ z\tau(z) \in S_1 $. Since
\[
    D(z-\tau(z))^2 = \left(\sqrt{D}(z-\tau(z))\right)^2 \in S_1,
\]
we have also $ (z-\tau(z))^2 \in S_1 $, and hence
\[
    z\tau(z) = \frac{1}{2}\left((z+\tau(z))^2-(z-\tau(z))^2\right) \in S_1.
\]
This proves the claim.

Let $ g \in \Gamma_K^1 $. We prove $ g \in S_1 $ by induction on $ d = \deg(g) $. If $ d = 0 $, then $ g(z) = a \in \Q $ is a constant polynomial. Thus we assume that $ d \geq 1 $ and that the $\Q$-valued generalized polynomials of degree $ < d $ are contained in $ S_1 $. We write $ g $ as
\[
    g(z) = \sum_{\substack{i, j \geq 0\\i+j \leq d}}\alpha_{i,j}z^i\tau(z)^j = \sum_{\substack{i, j \geq 1\\i+j \leq d}}\alpha_{i,j}z^i\tau(z)^j+\sum_{i=0}^d \alpha_{i,0}z^i+\sum_{j=0}^d \alpha_{0,j}\tau(z)^j.
\]
By~\Cref{propGenPol}, $ \alpha_{j,i} = \tau(\alpha_{i,j}) $. Let
\begin{align*}
    g_1(z)& = \sum_{\substack{i, j \geq 1\\i+j \leq d}}\alpha_{i,j}z^i\tau(z)^j,\\
    g_2(z)& = \sum_{i=0}^d \alpha_{i,0}z^i+\tau(\alpha_{i,0})\tau(z)^i,
\end{align*}
so that $ g(z) = g_1(z)+g_2(z) $. We have $ g_1(z) = z\tau(z)\cdot h_1(z) $ where
\[
    h_1(z) = \sum_{\substack{i, j \geq 1\\i+j \leq d}}\alpha_{i,j}z^{i-1}\tau(z)^{j-1}
\]
is a $ \Q $-valued generalized polynomial of degree $ <d $. By the inductive hypothesis, $ h_1 \in S_1 $, and because we also know that $ z\tau(z) \in S_1 $, we get $ g_1 \in S_1 $.

Next, we have $ g_2(z) = \alpha_{d,0}z^d+\tau(\alpha_{d,0})\tau(z)^d+h_2(z) $ where
\[
    h_2(z) = \sum_{i=0}^{d-1}\alpha_{i,0}z^i+\tau(\alpha_{i,0})\tau(z)^i
\]
is a $ \Q $-valued generalized polynomial of degree $ < d $. By the inductive hypothesis, $ h_2 \in S_2 $. We show $\alpha_{d,0}z^d+\tau(\alpha_{d,0})\tau(z)^d \in S_1 $. If we let $ \alpha_{d,0} = a+b\sqrt{D} $ for $ a, b \in \Q $, then
\[  \alpha_{d,0}z^d+\tau(\alpha_{d,0})\tau(z)^d = a(z^d+\tau(z)^d)+b\sqrt{D}(z^d-\tau(z)^d).
\]

We have $ z^d+\tau(z)^d = (z+\tau(z))^d-z\tau(z)\cdot r_1(z) $ where
\[
    r_1(z) = \sum_{i=1}^{d-1}\binom{d}{i}z^{i-1}\tau(z)^{d-i-1}
\]
is a $ \Q $-valued generalized polynomial of degree $ < d $, hence $ r_1 \in S_1 $. Because $ (z+\tau(z))^d $ and $ z\tau(z) $ also belong to $ S_1 $, we get $ z^d+\tau(z)^d \in S_1 $.

We also have $ \sqrt{D}(z^d-\tau(z)^d) = \sqrt{D}(z-\tau(z))\cdot r_2(z) $ where
\[
    r_2(z) = \sum_{i=0}^{d-1}z^{d-1-i}\tau(z)^i
\]
is a $ \Q $-valued generalized polynomial of degree $ < d $, hence $ r_2 \in S_1 $. Because $ \sqrt{D}(z-\tau(z))\in S_1 $, we get $ \sqrt{D}(z^d-\tau(z)^d) \in S_1 $. This proves $ \alpha_{d,0}z^d+\tau(\alpha_{d,0})\tau(z)^d \in S_1 $, thus $ g_2 \in S_1 $.

Finally, because $ g(z) = g_1(z)+g_2(z) $, we get $ g \in S_1 $.
\end{proof}

Next, we find a basis of the vector space $ \Gamma_K^{n,d} $ over $ \Q $. Let $ n \in \Z_{\geq 1} $ and $ d \in \Z_{\geq 0} $. We define the polynomials $ g_{\mathbf{i}\mathbf{j}} $ for $ \mathbf{i}, \mathbf{j} \in \Z_{\geq 0}^n $, $ |\mathbf{i}|+|\mathbf{j}| \leq d $ as follows: If $ \mathbf{i}<\mathbf{j} $, then
\begin{align*}
    g_{\mathbf{i}\mathbf{j}}(\mathbf{z})& = \mathbf{z}^\mathbf{i}\tau(\mathbf{z})^\mathbf{j}+\tau(\mathbf{z})^\mathbf{i}\mathbf{z}^\mathbf{j},\\
    g_{\mathbf{j}\mathbf{i}}(\mathbf{z})& = \sqrt{D}\left(\mathbf{z}^\mathbf{i}\tau(\mathbf{z})^\mathbf{j}-\tau(\mathbf{z})^\mathbf{i}\mathbf{z}^\mathbf{j}\right)
\end{align*}
and if $ \mathbf{i} = \mathbf{j} $, then
\[
    g_{\mathbf{i}\mathbf{i}} = \mathbf{z}^\mathbf{i}\tau(\mathbf{z})^\mathbf{i}.
\]

\begin{lemma}
\label{lemmaBasis2}
Let $ K = \Q(\sqrt{D}) $ where $ D \in \Z\setminus\{0,1\} $ is squarefree, $ n \in \Z_{\geq 1} $, and $ d \in \Z_{\geq 0} $. The elements $ g_{\mathbf{i}\mathbf{j}} $ for $ \mathbf{i}, \mathbf{j} \in \Z_{\geq 0}^n $, $ |\mathbf{i}|+|\mathbf{j}|\leq d $ form a basis of the vector space $ \Gamma_K^{n,d} $ over $ \Q $.
\end{lemma}
\begin{proof}
Let
\[
    g(\mathbf{z}) = \sum_{|\mathbf{i}|+|\mathbf{j}|\leq d} \alpha_{\mathbf{i}\mathbf{j}}\mathbf{z}^\mathbf{i}\tau(\mathbf{z})^\mathbf{j} \in \Gamma_K^{n,d}.
\]
By~\Cref{propGenPol}, $ \alpha_{\mathbf{j}\mathbf{i}} = \tau(\alpha_{\mathbf{i}\mathbf{j}}) $ for every $ \mathbf{i}, \mathbf{j} \in \Z_{\geq 0}^n $, hence
\[
    g(\mathbf{z}) = \sum_{\substack{\mathbf{i}<\mathbf{j}\\|\mathbf{i}|+|\mathbf{j}|\leq d}}\alpha_{\mathbf{i}\mathbf{j}}\mathbf{z}^\mathbf{i}\tau(\mathbf{z})^\mathbf{j}+\tau(\alpha_{\mathbf{i}\mathbf{j}})\tau(\mathbf{z})^\mathbf{i}\mathbf{z}^\mathbf{j}+\sum_{2|\mathbf{i}|\leq d}\alpha_{\mathbf{i}\mathbf{i}}\mathbf{z}^\mathbf{i}\tau(\mathbf{z})^\mathbf{i}.
\]
Let $ \alpha_{\mathbf{i}\mathbf{j}} = a_{\mathbf{i}\mathbf{j}}+b_{\mathbf{i}\mathbf{j}}\sqrt{D} $ where $ a_{\mathbf{i}\mathbf{j}}, b_{\mathbf{i}\mathbf{j}} \in \Q $. In particular, $ \alpha_{\mathbf{i}\mathbf{i}} = \tau(\alpha_{\mathbf{i}\mathbf{i}}) $ implies $ \alpha_{\mathbf{i}\mathbf{i}} = a_{\mathbf{i}\mathbf{i}} $. We get
\begin{align*}
    g(\mathbf{z})& = \sum_{\substack{\mathbf{i}<\mathbf{j}\\|\mathbf{i}|+|\mathbf{j}|\leq d}}a_{\mathbf{i}\mathbf{j}}\left(\mathbf{z}^\mathbf{i}\tau(\mathbf{z})^\mathbf{j}+\tau(\mathbf{z})^\mathbf{i}\mathbf{z}^\mathbf{j}\right)+b_{\mathbf{i}\mathbf{j}}\sqrt{D}\left(\mathbf{z}^\mathbf{i}\tau(\mathbf{z})^\mathbf{j}-\tau(\mathbf{z})^\mathbf{i}\mathbf{z}^\mathbf{j}\right)+\sum_{2|\mathbf{i}|\leq d}a_{\mathbf{i}\mathbf{i}}\mathbf{z}^\mathbf{i}\tau(\mathbf{z})^\mathbf{i},
\end{align*}
hence $ g $ is a $\Q$-linear combination of the elements $ g_{\mathbf{i}\mathbf{j}} $.

Let
\[
    B = \{g_{\mathbf{i}\mathbf{j}}:\; \mathbf{i}, \mathbf{j}\in \Z_{\geq 0}^n, |\mathbf{i}|+|\mathbf{j}|\leq d\}.
\]
If $ \mathbf{i} \neq \mathbf{j} $, then the only generalized polynomials in $ B $ containing the term $ \mathbf{z}^\mathbf{i}\tau(\mathbf{z})^\mathbf{j} $ are $ g_{\mathbf{i}\mathbf{j}} $ and $ g_{\mathbf{j}\mathbf{i}} $. Since they are linearly independent, neither can be expressed as a linear combination of the remaining elements in $ B $. Similarly, if $ \mathbf{i} = \mathbf{j} $, then $ \mathbf{z}^\mathbf{i}\tau(\mathbf{z})^\mathbf{i} $ is contained only in $ g_{\mathbf{i}\mathbf{i}} $, and thus $ g_{\mathbf{i}\mathbf{i}} $ cannot be expressed as a linear combination of the other elements in $ B $. This proves that the $ g_{\mathbf{i}\mathbf{j}} $ are linearly independent.
\end{proof}

\begin{proof}[Proof of \Cref{thmBasis}]
Let
\[
    S_n = \Q\left[z_1+\tau(z_1), \sqrt{D}(z_1-\tau(z_1)), \dots, z_n+\tau(z_n),\sqrt{D}(z_n-\tau(z_n))\right].
\]
We have $ S_n \subset \Gamma_K^n $ because the generators of $ S_n $ are $ \Q $-valued, and we want to prove $ \Gamma_K^n \subset S_n $.

We proceed by induction on $ n $. For $ n = 1 $, this is \Cref{lemmaBasis1}. Suppose that $ n \geq 2 $ and the statement is true for polynomials with $ < n $ variables.

Let $ g \in \Gamma_K^n $. We show $ g \in S_n $ by induction on $ d = \deg(g) $. If $ d = 0 $, then $ g(\mathbf{z}) = a \in \Q $ is a constant polynomial. Thus we assume that $ d \geq 1 $ and that the $ \Q $-valued generalized polynomials of degree $ < d $ in $ n $ variables are contained in $ S_n $. The polynomial $ g $ is a $\Q$-linear combination of the elements in the basis from \Cref{lemmaBasis2}. Thus it is enough to show that the basis elements of degree $ d $ belong to $ S_n $.

\

We show that if $ \mathbf{i}, \mathbf{j} \in \Z_{\geq 0}^n $, $ |\mathbf{i}|+|\mathbf{j}| = d $, and $ \mathbf{i}<\mathbf{j} $, then
\begin{align*}
    g_{\mathbf{i}\mathbf{j}}(\mathbf{z})& = \mathbf{z}^\mathbf{i}\tau(\mathbf{z})^\mathbf{j}+\tau(\mathbf{z})^\mathbf{i}\mathbf{z}^\mathbf{j} \in S_n,\\
    g_{\mathbf{j}\mathbf{i}}(\mathbf{z})& = \sqrt{D}\left(\mathbf{z}^\mathbf{i}\tau(\mathbf{z})^\mathbf{j}-\tau(\mathbf{z})^\mathbf{i}\mathbf{z}^\mathbf{j}\right) \in S_n.
\end{align*}

If $ i_n =0 $ and $ j_n = 0 $, then the variable $ z_n $ does not appear in $ g_{\mathbf{i}\mathbf{j}} $ and $ g_{\mathbf{j}\mathbf{i}} $, hence $ g_{\mathbf{i}\mathbf{j}}, g_{\mathbf{j}\mathbf{i}} \in S_n $ by the inductive hypothesis.

If $ i_n \geq 1 $, then we let
\begin{align*}
    h(\mathbf{z})& = z_1^{i_1}\tau(z_1)^{j_1}z_2^{i_2}\tau(z_2)^{j_2}\cdots z_n^{i_n-1}\tau(z_n)^{j_n},\\
    g^1_{\mathbf{i}\mathbf{j}}(\mathbf{z})& = (z_n+\tau(z_n))\cdot (h(\mathbf{z})+\tau(h(\mathbf{z}))),\\
    g^2_{\mathbf{i}\mathbf{j}}(\mathbf{z})& = (z_n-\tau(z_n))\cdot (h(\mathbf{z})-\tau(h(\mathbf{z}))),\\
    g^1_{\mathbf{j}\mathbf{i}}(\mathbf{z})& = \sqrt{D}(z_n-\tau(z_n))\cdot (h(\mathbf{z})+\tau(h(\mathbf{z}))),\\
    g^2_{\mathbf{j}\mathbf{i}}(\mathbf{z})& = (z_n+\tau(z_n))\cdot \sqrt{D}(h(\mathbf{z})-\tau(h(\mathbf{z}))).
\end{align*}
so that
\begin{align*}
    g_{\mathbf{i}\mathbf{j}}(\mathbf{z})& = z_n h(\mathbf{z})+\tau(z_n)\tau(h(\mathbf{z})) = \frac{1}{2}\left(g^1_{\mathbf{i}\mathbf{j}}+g^2_{\mathbf{i}\mathbf{j}}\right),\\
    g_{\mathbf{j}\mathbf{i}}(\mathbf{z})& = \sqrt{D}\left(z_nh(\mathbf{z})-\tau(z_n)\tau(h(\mathbf{z}))\right) = \frac{1}{2}\left(g^1_{\mathbf{j}\mathbf{i}}+g^2_{\mathbf{j}\mathbf{i}}\right).
\end{align*}
Because $ h(\mathbf{z}) $ has degree $ d-1 $, we can apply induction on $ h(\mathbf{z})+\tau(h(\mathbf{z})) $ and $ \sqrt{D}(h(\mathbf{z})-\tau(h(\mathbf{z}))) $. Since
\[
    z_n+\tau(z_n),\ \sqrt{D}(z_n-\tau(z_n)),\ h(\mathbf{z})+\tau(h(\mathbf{z})),\ \sqrt{D}(h(\mathbf{z})-\tau(h(\mathbf{z}))) \in S_n,
\]
we immediately obtain $ g^1_{\mathbf{i}\mathbf{j}}, g^1_{\mathbf{j}\mathbf{i}}, g^2_{\mathbf{j}\mathbf{i}} \in S_n $. We also get
\[
    g^2_{\mathbf{i}\mathbf{j}} = \frac{1}{D}\cdot\sqrt{D}(z_n-\tau(z_n))\cdot\sqrt{D}(h(\mathbf{z})-\tau(h(\mathbf{z}))) \in S_n,
\]
and $ g_{\mathbf{i}\mathbf{j}}, g_{\mathbf{j}\mathbf{i}} \in S_n $ follows.

If $ i_n = 0 $ and $ j_n \geq 1 $, then the proof is analogous to the case $ i_n \geq 1 $.

\

Next, we show that if $ \mathbf{i} \in \Z_{\geq 0}^n $ and $ 2|\mathbf{i}| = d $, then $ g_{\mathbf{i}\mathbf{i}}(\mathbf{z}) = \mathbf{z}^\mathbf{i}\tau(\mathbf{z})^\mathbf{i} \in S_n $.

If $ i_n = 0 $, then $ \mathbf{g}_{ii} $ is a polynomial in $ n-1 $ variables, hence $ \mathbf{g}_{ii} \in S_n $ by the inductive hypothesis.

If $ i_n \geq 1 $, then we let
\[
    h(\mathbf{z}) = z_1^{i_1}\tau(z_1)^{i_1}z_2^{i_2}\tau(z_2)^{i_2}\cdots z_n^{i_n-1}\tau(z_n)^{i_n-1}.
\]
We have $ g_{\mathbf{i}\mathbf{i}}(\mathbf{z}) = z_n\tau(z_n)\cdot h(\mathbf{z}) $. Since the theorem holds for generalized polynomials in one variable, we get
\[
    z_n\tau(z_n) \in \Q\left[z_n+\tau(z_n), \sqrt{D}\left(z_n-\tau(z_n)\right)\right] \subset S_n.
\]
The generalized polynomial $ h $ has degree $ d-2 $, hence $ h \in S_n $ by the inductive hypothesis, and $ g_{\mathbf{i}\mathbf{i}} \in S_n $ follows.
\end{proof}

\section{Generalized quadratic forms}
\label{secGenForms}

Generalized quadratic forms are defined in \cite[Definition~1.1]{BPS} over a Galois extension $ K/\Q $ of degree $ d $ (which is assumed to be totally real). We need the definition only over quadratic fields.

\begin{definition}
\label{defGenQF}
Let $ K = \Q(\sqrt{D}) $ where $ D \in \Z\setminus\{0,1\} $ is squarefree. A \emph{generalized quadratic form} in $ n $ variables over $ K $ is a generalized polynomial
\begin{equation}
    \label{eqGenQF}
    G(z_1, z_2, \dots, z_n) = \sum_{1\leq i \leq j \leq n} \alpha_{ij}z_iz_j+\sum_{1\leq i, j \leq n} \beta_{ij}z_i\tau(z_j)+\sum_{1 \leq i \leq j \leq n}\gamma_{ij}\tau(z_i)\tau(z_j)
\end{equation}
where $ \alpha_{ij}, \gamma_{ij} \in K $ for $ 1 \leq i \leq j \leq n $ and $ \beta_{ij} \in K $ for $ 1 \leq i, j \leq n $.
\end{definition}

A generalized quadratic form $ G $ given by~\eqref{eqGenQF} will be called \emph{integral} if $ \alpha_{ij}, \gamma_{ij} \in \O_K $ for $ 1 \leq i \leq j \leq n $ and $ \beta_{ij} \in \O_K $ for $ 1 \leq i, j \leq n $. An integral generalized quadratic form will be called \emph{classical} if $ \alpha_{ij}, \gamma_{ij} \in 2\O_K $ for $ i < j $ and $ \beta_{ij} \in 2\O_K $ for $ 1 \leq i, j \leq n $. 
\medskip

Let us recall that a quadratic form $ Q $ in $ n $ variables over $ K $ is a homogeneous polynomial of degree $ 2 $ over $K$. We can also obtain it as a generalized quadratic form with $\beta_{ij},\gamma_{ij}=0$ for all $1\leq i,j \leq n.$

Then if $ G $ is a generalized quadratic form in $ n $ variables over $ K $, then we can define a \emph{quadratic form $ Q $ associated to $ G $} in $ 2n $ variables by
\begin{equation}
\label{eqDefQ}
    Q(x_1, y_1, \dots, x_n, y_n) = G(x_1+y_1\omega_D, \dots, x_n+y_n\omega_D).
\end{equation}
In other words, $ Q $ is obtained from $ G $ by expressing the variables $ z_1, z_2, \dots, z_n $ in the integral basis $ (1, \omega_D) $. It has coefficients in $ K $ and the variables $ x_1, y_1, \dots, x_n, y_n $ take values in $ \Q $.

If $ G $ is a generalized form in $ n $ variables over $ K $ given by \eqref{eqGenQF}, then we define the \emph{matrix of $ G $} as the $ 2n \times 2n $ matrix
\[
    M_G = \begin{pmatrix}
        A&B\\
        B^\top&C
    \end{pmatrix}
\]
where
\[
    A = \begin{pmatrix}
        \alpha_{11}&\frac{\alpha_{12}}{2}&\dots&\frac{\alpha_{1n}}{2}\\
        \frac{\alpha_{12}}{2}&\alpha_{22}&\dots&\frac{\alpha_{2n}}{2}\\
        \vdots&\vdots&\ddots&\vdots\\
        \frac{\alpha_{1n}}{2}&\frac{\alpha_{2n}}{2}&\dots&\alpha_{nn}
    \end{pmatrix},\ 
    B = \begin{pmatrix}
        \frac{\beta_{11}}{2}&\frac{\beta_{12}}{2}&\dots&\frac{\beta_{1n}}{2}\\
        \frac{\beta_{21}}{2}&\frac{\beta_{22}}{2}&\dots&\frac{\beta_{2n}}{2}\\
        \vdots&\vdots&\ddots&\vdots\\
        \frac{\beta_{n1}}{2}&\frac{\beta_{n2}}{2}&\dots&\frac{\beta_{nn}}{2}
    \end{pmatrix},\ 
    C = \begin{pmatrix}
        \gamma_{11}&\frac{\gamma_{12}}{2}&\dots&\frac{\gamma_{1n}}{2}\\
        \frac{\gamma_{12}}{2}&\gamma_{22}&\dots&\frac{\gamma_{2n}}{2}\\
        \vdots&\vdots&\ddots&\vdots\\
        \frac{\gamma_{1n}}{2}&\frac{\gamma_{2n}}{2}&\dots&\gamma_{nn}
    \end{pmatrix}.
\]
Thus $ M_G $ is the matrix such that
\[
    G(\mathbf{z}) = (\mathbf{z},\tau(\mathbf{z}))M_G\begin{pmatrix}\mathbf{z}\\\tau(\mathbf{z})\end{pmatrix}
\]
for $ \mathbf{z} \in K^n $. The \emph{rank} of $ G $ is the rank of the matrix $ M_G$ over $ K $.

The $ 2n \times 2n $ matrix
\[
    T = \begin{pmatrix}
        1&\omega_D&0&0&\dots&0&0\\
        0&0&1&\omega_D&\dots&0&0\\
        \vdots&\vdots&\vdots&\ddots&\vdots&\vdots&\vdots\\
        0&0&\dots&0&0&1&\omega_D\\
        1&\tau(\omega_D)&0&0&\dots&0&0\\
        0&0&1&\tau(\omega_D)&\dots&0&0\\
        \vdots&\vdots&\vdots&\ddots&\vdots&\vdots&\vdots\\
        0&0&\dots&0&0&1&\tau(\omega_D)\\
    \end{pmatrix}
\]
determines the change of variables $ (x_1, y_1, \dots, x_n, y_n) \mapsto (z_1, \dots, z_n, \tau(z_1), \dots, \tau(z_n)) $. Thus, if is the quadratic form $ Q $ is associated to the quadratic form $G$ defined by~\eqref{eqDefQ}, and $M_Q, M_G$ are their respective matrices, then
\[
    M_Q = T^\top M_G T.
\]

If $ G $ is a generalized quadratic form over $ \Q $ then we say that $ G $ \emph{represents} $ b \in \Z $ over $ \Z $ if there exists $ \mathbf{a} \in \Z^n $ such that $ G(\mathbf{a}) = b $. In particular, according to our definition, every quadratic form represents $ 0 $. If there exists $ \mathbf{a} \in \Z^n\setminus\{(0,0,\dots,0)\} $ such that $ G(\mathbf{a}) = 0 $, then we say that $ G $ represents $ 0 $ \emph{non-trivially}. The form $ G $ is called \emph{positive definite} if $ M_G $ is a positive definite matrix and $ G $ is called \emph{universal} if it is a positive definite integral form such that it represents every $ b \in \Z_{\geq 1} $.

\begin{lemma}
\label{lemmaDetMQ}
Let $ K = \Q(\sqrt{D}) $ where $ D \in \Z\setminus\{0,1\} $ is squarefree, $ G $ be a generalized form in $ n $ variables over $ K $ given by \eqref{eqGenQF}, and $ M_G $ be the matrix of $ G $. If $ Q $ is the quadratic form in $ 2n $ variables associated to $ G $ and $ M_Q $ is the matrix of $ Q $, then
\[
    \det(M_Q) = D^n\cdot\begin{cases}
        \det(2M_G),&\text{if }D\equiv 2,3 \pmod{4},\\
        \det(M_G),&\text{if }D\equiv 1 \pmod{4}.
    \end{cases}
\]
\end{lemma}
\begin{proof}
Since $ M_G = T^\top M_Q T $, we get
\[
    \det(M_Q) = (\det T)^2\det(M_G).
\]

Let us compute the determinant of $ T $. If $ \pi $ is the permutation
\[
    \pi = \begin{pmatrix}
        1&2&3&4&\dots&2n-1&2n\\
        1&n+1&2&n+2&\dots&n&2n
    \end{pmatrix}
\]
(the out-shuffle), then $ \sgn(\pi) = (-1)^\frac{n(n-1)}{2} $. Applying this permutation to the rows of $ T $, we obtain a block-diagonal matrix where each block equals
\[
    M = \begin{pmatrix}
        1&\omega_D\\
        1&\tau(\omega_D)
    \end{pmatrix}.
\]
If $ D \equiv 2, 3 \pmod{4} $, then
\[
    \det M = \det\begin{pmatrix}
        1&\sqrt{D}\\
        1&-\sqrt{D}
    \end{pmatrix} = -2\sqrt{D},
\]
while if $ D \equiv 1 \pmod{4} $, then
\[
    \det M = \det\begin{pmatrix}
        1&\frac{1+\sqrt{D}}{2}\\
        1&\frac{1-\sqrt{D}}{2}
    \end{pmatrix} = -\sqrt{D}.
\]
Thus
\[
    \det T = (-1)^\frac{n(n-1)}{2}(\det M)^n = \begin{cases}
        (-1)^\frac{n(n+1)}{2}D^\frac{n}{2}2^n,&\text{if }D \equiv 2,3\pmod{4},\\
        (-1)^\frac{n(n+1)}{2}D^\frac{n}{2},&\text{if }D \equiv 1 \pmod{4}.
    \end{cases}
\]

In the case $ D \equiv 2, 3 \pmod{4} $, we get
\[
    \det(M_Q) = D^n 2^{2n} \det(M_G) = D^n \det(2M_G),
\]
and in the case $ D \equiv 1 \pmod{4} $, we get
\[
    \det(M_Q) = D^n \det(M_G).\qedhere
\]
\end{proof}

\begin{corollary}
\label{corDetMQ}
Let $ K = \Q(\sqrt{D}) $ where $ D \in \Z\setminus\{0,1\} $ is squarefree, $ G $ be a generalized form in $ n $ variables over $ K $ given by \eqref{eqGenQF}, $ Q $ be the quadratic form in $ 2n $ variables associated to $ G $, and $ M_Q $ be the matrix of $ Q $. If one of the following conditions is satisfied:
\begin{enumerate}[i)]
    \item $ D \equiv 2, 3 \pmod{4} $ and $ G $ is integral,
    \item $ D \equiv 1 \pmod{4} $ and $ G $ is classical,
\end{enumerate}
then $ D^n \mid \det(M_Q) $.
\end{corollary}
\begin{proof}
First, assume that $ D \equiv 2, 3 \pmod{4} $ and $ G $ is integral. Since the elements of $ M_G $ are in $ \frac{1}{2}\O_K $, we must have $ \det(2M_G) \in \O_K $. Thus $ D^n \mid \det(M_Q) $ by \Cref{lemmaDetMQ}.

Secondly, assume that $ D \equiv 1 \pmod{4} $ and $ G $ is classical. The elements of $ M_G $ are in $ \O_K $, hence $ \det(M_G) \in \O_K $, and $ D^n \mid \det(M_Q) $ by \Cref{lemmaDetMQ}.
\end{proof}

\begin{lemma}
\label{lemmaMQ}
Let $ K = \Q(\sqrt{D}) $ where $ D \in \Z\setminus\{0,1\} $ is squarefree and let $ G $ be a generalized form in $ n $ variables over $ K $ given by \eqref{eqGenQF}. Assume that $ G(\mathbf{a}) \in \Q $ for every $ \mathbf{a} \in K^n $. Let $ Q $ be the quadratic form in $ 2n $ variables associated to $ G $. If $ 1 \leq i \leq j \leq n $, then we let $ a_{ij} $, $ b_{ij} $, $ b_{ji} $, and $ c_{ij} $ be the coefficients of the terms $ x_ix_j $, $ x_iy_j $, $ x_jy_i $, and $ y_iy_j $ in $ Q $, respectively, so that the matrix $ M_Q $ of the quadratic form $ Q $ is
\[
    M_Q = \begin{pmatrix}
        \begin{matrix}
            a_{11}&\frac{b_{11}}{2}\\
            \frac{b_{11}}{2}&c_{11}
        \end{matrix}&
        \begin{matrix}
            \frac{a_{12}}{2}&\frac{b_{12}}{2}\\
            \frac{b_{21}}{2}&\frac{c_{12}}{2}
        \end{matrix}&\dots&
        \begin{matrix}
            \frac{a_{1n}}{2}&\frac{b_{1n}}{2}\\
            \frac{b_{n1}}{2}&\frac{c_{1n}}{2}
        \end{matrix}\\
        \begin{matrix}
            \frac{a_{12}}{2}&\frac{b_{21}}{2}\\
            \frac{b_{12}}{2}&\frac{c_{12}}{2}
        \end{matrix}&
        \begin{matrix}
            a_{22}&\frac{b_{22}}{2}\\
            \frac{b_{22}}{2}&c_{22}
        \end{matrix}&\dots&
        \begin{matrix}
            \frac{a_{2n}}{2}&\frac{b_{2n}}{2}\\
            \frac{b_{n2}}{2}&\frac{c_{2n}}{2}
        \end{matrix}\\
        \vdots&\vdots&\ddots&\vdots\\
        \begin{matrix}
            \frac{a_{1n}}{2}&\frac{b_{n1}}{2}\\
            \frac{b_{1n}}{2}&\frac{c_{1n}}{2}
        \end{matrix}&
        \begin{matrix}
            \frac{a_{2n}}{2}&\frac{b_{n2}}{2}\\
            \frac{b_{2n}}{2}&\frac{c_{2n}}{2}
        \end{matrix}&\dots&
        \begin{matrix}
            a_{nn}&\frac{b_{nn}}{2}\\
            \frac{b_{nn}}{2}&c_{nn}
        \end{matrix}
    \end{pmatrix}.
\]
If $ 1 \leq i < j \leq n $, then
\begin{align*}
    a_{ij}& = \Tr(\alpha_{ij})+\Tr(\beta_{ij}),\\
    b_{ij}& = \Tr(\alpha_{ij}\omega_D)+\Tr(\beta_{ij}\tau(\omega_D)),\\
    b_{ji}& = \Tr(\alpha_{ij}\omega_D)+\Tr(\beta_{ij}\omega_D),\\
    c_{ij}& = \Tr(\alpha_{ij}\omega_D^2)+\Tr(\beta_{ij})\Nm(\omega_D).
\end{align*}
Moreover, if $ 1 \leq i \leq n $, then
\begin{align*}
    a_{ii}& = \Tr(\alpha_{ii})+\beta_{ii},\\
    b_{ii}& = \Tr(2\alpha_{ii}\omega_D)+\beta_{ii}\Tr(\omega_D),\\
    c_{ii}& = \Tr(\alpha_{ii}\omega_D^2)+\beta_{ii}\Nm(\omega_D).
\end{align*}
\end{lemma}
\begin{proof}
The quadratic form $ Q $ is given by
\begin{align*}
    Q(x_1, y_1, \dots, x_n, y_n)& = \sum_{1 \leq i \leq j \leq n}\alpha_{ij}(x_i+y_i\omega)(x_j+y_j\omega)+\sum_{1 \leq i,j \leq n}\beta_{ij}(x_i+y_i\omega)(x_j+y_j\tau(\omega))\\
    &+\sum_{1\leq i \leq j \leq n}\gamma_{ij}(x_i+y_i\tau(\omega))(x_j+y_j\tau(\omega)).
\end{align*}

By~\Cref{propGenPol}, $ G(\mathbf{a}) \in \Q $ for every $ \mathbf{a} \in K^n $ if and only if the coefficients of conjugate terms are conjugate, i.e., $ \gamma_{ij} = \tau(\alpha_{ij}) $ for $ 1 \leq i \leq j \leq n $ and $ \beta_{ji} = \tau(\beta_{ij}) $ for $ 1 \leq i, j \leq n $. If $ i < j $, then
\begin{align*}
    a_{ij}& = \alpha_{ij}+\beta_{ij}+\beta_{ji}+\gamma_{ij} = \Tr(\alpha_{ij})+\Tr(\beta_{ij}),\\
    b_{ij}& = \alpha_{ij}\omega+\beta_{ij}\tau(\omega)+\beta_{ji}\omega+\gamma_{ij}\tau(\omega) = \Tr(\alpha_{ij}\omega)+\Tr(\beta_{ij}\tau(\omega)),\\
    b_{ji}& = \alpha_{ij}\omega+\beta_{ij}\omega+\beta_{ji}\tau(\omega)+\gamma_{ij}\tau(\omega) = \Tr(\alpha_{ij}\omega)+\Tr(\beta_{ij}\omega),\\
    c_{ij}& = \alpha_{ij}\omega^2+(\beta_{ij}+\beta_{ji})\Nm(\omega)+\gamma_{ij}\tau(\omega)^2 = \Tr(\alpha_{ij}\omega^2)+\Tr(\beta_{ij})\Nm(\omega).
\end{align*}
Moreover,
\begin{align*}
    a_{ii}& = \alpha_{ii}+\beta_{ii}+\gamma_{ii} = \Tr(\alpha_{ii})+\beta_{ii},\\
    b_{ii}& = 2\alpha_{ii}\omega+\beta_{ii}\Tr(\omega)+2\gamma_{ii}\tau(\omega) = \Tr(2\alpha_{ii}\omega)+\beta_{ii}\Tr(\omega),\\
    c_{ii}& = \alpha_{ii}\omega^2+\beta_{ii}\Nm(\omega)+\gamma_{ii}\tau(\omega)^2 = \Tr(\alpha_{ii}\omega^2)+\beta_{ii}\Nm(\omega).\qedhere
\end{align*}
\end{proof}

\begin{lemma}
\label{lemmaGQ}
Let $ K = \Q(\sqrt{D}) $ where $ D \in \Z \setminus \{0, 1\} $ is squarefree and let $ G $ be a generalized form in $ n $ variables over $ K $ given by~\eqref{eqGenQF}. Assume that $ G $ is $ \Z $-valued. Let $ Q $ be the quadratic form in $ 2n $ variables associated to $ G $. If $ G $ is integral, then $ Q $ is integral. Moreover, if one of the following conditions holds:
\begin{enumerate}[(i)]
    \item $ D \equiv 2, 3 \pmod{4} $ and $ G $ is integral,
    \item $ D \equiv 1 \pmod{4} $ and $ G $ is classical,
\end{enumerate}
then $ Q $ is classical.
\end{lemma}
\begin{proof}
As before, we let $ a_{ij} $, $ b_{ij} $, $ b_{ji} $, and $ c_{ij} $ be the coefficients of the terms $ x_ix_j $, $ x_iy_j $, $ x_jy_i $, and $ y_iy_j $ in $ Q $, respectively, for $ 1 \leq i \leq j \leq n $. If $ G $ is integral, then $ \alpha_{ij}, \beta_{ij}\in \O_K $ for $ i \leq j $. The elements of $ \O_K $ have trace in $ \Z $, hence $ a_{ij}, b_{ij}, b_{ji}, c_{ij} \in \Z $ by \Cref{lemmaMQ} and $ Q $ is integral.

If $ G $ is classical, then $ \alpha_{ij}, \beta_{ij} \in 2\O_K $ for $ i < j $ and $ \beta_{ii} \in 2\Z $. It follows from \Cref{lemmaMQ} that $ a_{ij}, b_{ij}, b_{ji}, c_{ij} \in 2\Z $ for $ i < j $ and $ b_{ii} \in 2\Z $, hence $ Q $ is classical.

If $ D \equiv 2, 3 \pmod{4} $, then $ \Tr(\alpha) $ is even for every $ \alpha \in \O_K $. Thus it is sufficient to assume that $ G $ is integral to conclude that $ Q $ is classical.
\end{proof}

\begin{lemma}
\label{lemmaMQn}
Let $ K = \Q(\sqrt{D}) $ where $ D \in \Z \setminus\{0,1\} $ is squarefree and let $ G $ be a generalized form in $ n $ variables given by \eqref{eqGenQF}. Assume that $ G(\mathbf{a}) \in \Q $ for every $ \mathbf{a} \in K^n $. If $ 1 \leq i \leq j \leq n $, then we let $ \alpha_{ij} = r_{ij}+s_{ij}\omega_D $ and $ \beta_{ij} = t_{ij}+u_{ij}\omega_D $ where $ r_{ij}, s_{ij}, t_{ij}, u_{ij} \in \Q $. Let $ Q $ be the quadratic form in $ 2n $ variables associated to $ G $. If $ 1 \leq i \leq j \leq n $, then we let $ a_{ij} $, $ b_{ij} $, $ b_{ji} $, and $ c_{ij} $ be the coefficients of the terms $ x_ix_j $, $ x_iy_j $, $ x_jy_i $, and $ y_iy_j $ in $ Q $, respectively.

If $ D \equiv 2, 3 \pmod{4} $ and $ 1 \leq i < j \leq n $, then
\begin{align*}
    a_{ij}& = 2r_{ij}+2t_{ij},\\
    b_{ij}& = 2Ds_{ij}-2Du_{ij},\\
    b_{ji}& = 2Ds_{ij}+2Du_{ij},\\
    c_{ij}& = 2Dr_{ij}-2Dt_{ij}.
\end{align*}
Moreover, if $ 1 \leq i \leq n $, then
\begin{align*}
    a_{ii}& = 2r_{ii}+t_{ii},\\
    b_{ii}& = 4Ds_{ii},\\
    c_{ii}& = 2Dr_{ii}-Dt_{ii}.
\end{align*}

If $ D \equiv 1 \pmod{4} $ and $ 1 \leq i < j \leq n $, then
\begin{align*}
    a_{ij}& = 2r_{ij}+s_{ij}+2t_{ij}+u_{ij},\\
    b_{ij}& = r_{ij}+\frac{1+D}{2}s_{ij}+t_{ij}+\frac{1-D}{2}u_{ij},\\
    b_{ji}& = r_{ij}+\frac{1+D}{2}s_{ij}+t_{ij}+\frac{1+D}{2}u_{ij},\\
    c_{ij}& = \frac{1+D}{2}r_{ij}+\frac{1+3D}{4}s_{ij}+\frac{1-D}{2}t_{ij}+\frac{1-D}{4}u_{ij}.
\end{align*}
Moreover, if $ 1 \leq i \leq n $, then
\begin{align*}
    a_{ii}& = 2r_{ii}+s_{ii}+t_{ii},\\
    b_{ii}& = 2r_{ii}+(1+D)s_{ii}+t_{ii},\\
    c_{ii}& = \frac{1+D}{2}r_{ii}+\frac{1+3D}{4}s_{ii}+\frac{1-D}{4}t_{ii}.
\end{align*}
\end{lemma}
\begin{proof}
We express the coefficients of $ Q $ using \Cref{lemmaMQ}. If $ 1 \leq i < j \leq n $, then
\begin{align*}
    a_{ij}& = \Tr(\alpha_{ij})+\Tr(\beta_{ij}) = \Tr(r_{ij}+s_{ij}\omega_D)+\Tr(t_{ij}+u_{ij}\omega_D)\\
    & = 2r_{ij}+s_{ij}\Tr(\omega_D)+2t_{ij}+u_{ij}\Tr(\omega_D),\\
    b_{ij}& = \Tr(\alpha_{ij}\omega_D)+\Tr(\beta_{ij}\tau(\omega_D)) = \Tr(r_{ij}\omega_D+s_{ij}\omega_D^2)+\Tr(t_{ij}\tau(\omega_D)+u_{ij}\Nm(\omega_D))\\
    & = r_{ij}\Tr(\omega_D)+s_{ij}\Tr(\omega_D^2)+t_{ij}\Tr(\omega_D)+2u_{ij}\Nm(\omega_D),\\
    b_{ji}& = \Tr(\alpha_{ij}\omega_D)+\Tr(\beta_{ij}\omega_D) = \Tr(r_{ij}\omega_D+s_{ij}\omega_D^2)+\Tr(t_{ij}\omega_D+u_{ij}\omega_D^2)\\
    & = r_{ij}\Tr(\omega_D)+s_{ij}\Tr(\omega_D^2)+t_{ij}\Tr(\omega_D)+u_{ij}\Tr(\omega_D^2),\\
    c_{ij}& = \Tr(\alpha_{ij}\omega_D^2)+\Tr(\beta_{ij})\Nm(\omega_D) = \Tr(r_{ij}\omega_D^2+s_{ij}\omega_D^3)+\Tr(t_{ij}+u_{ij}\omega_D)\Nm(\omega_D)\\& = r_{ij}\Tr(\omega_D^2)+s_{ij}\Tr(\omega_D^3)+2t_{ij}\Nm(\omega_D)+u_{ij}\Tr(\omega_D)\Nm(\omega_D).
\end{align*}
Moreover, if $ 1 \leq i \leq n $, then
\begin{align*}
    a_{ii}& = \Tr(\alpha_{ii})+\beta_{ii} = \Tr(r_{ii}+s_{ii}\omega_D)+t_{ii} = 2r_{ii}+s_{ii}\Tr(\omega_D)+t_{ii},\\
    b_{ii}& = \Tr(2\alpha_{ii}\omega_D)+\beta_{ii}\Tr(\omega_D) = \Tr(2r_{ii}\omega_D+2s_{ii}\omega_D^2)+t_{ii}\Tr(\omega_D) = 2r_{ii}\Tr(\omega_D)+2s_{ii}\Tr(\omega_D^2)+t_{ii}\Tr(\omega_D),\\
    c_{ii}& = \Tr(\alpha_{ii}\omega_D^2)+\beta_{ii}\Nm(\omega_D) = \Tr(r_{ii}\omega_D^2+s_{ii}\omega_D^3)+t_{ii}\Nm(\omega_D) = r_{ii}\Tr(\omega_D^2)+s_{ii}\Tr(\omega_D^3)+t_{ii}\Nm(\omega_D).
\end{align*}

If $ D \equiv 2, 3 \pmod{4} $, then $ \omega_D = \sqrt{D} $, hence $ \Tr(\omega_D) = 0 $, $ \Tr(\omega_D^2) = \Tr(D) = 2D $, $ \Tr(\omega_D^3) = \Tr(D\sqrt{D}) = 0 $, and $ \Nm(\omega_D) = -D $.

If $ D \equiv 1 \pmod{4} $, then $ \omega_D = \frac{1+\sqrt{D}}{2} $, hence $ \Tr(\omega_D) = 1 $, $ \Tr(\omega_D^2) = \Tr\left(\frac{1+D}{4}+\frac{\sqrt{D}}{2}\right) = \frac{1+D}{2} $, $ \Tr(\omega_D^3) = \Tr\left(\frac{1+3D}{8}+\frac{3+D}{8}\sqrt{D}\right) = \frac{1+3D}{4} $, and $ \Nm(\omega_D) = \frac{1-D}{4} $.

After substituting these expressions into the formulas above, we get the lemma.
\end{proof}

If $ p $ is a prime and $ M $ a matrix with entries in $ \Z $, then we let $ M \bmod p $ denote the matrix with entries in $ \Z/p\Z $ obtained by reducing the elements of $ M $ modulo $ p $.

\begin{proposition}
\label{propRank}
Let $ K = \Q(\sqrt{D}) $ where $ D \in \Z\setminus\{0,1\} $ is squarefree, $ p \mid D $ be a prime number, and $ G $ an integral generalized quadratic form in $ n $ variables over $ K $. Assume that $ G $ is $ \Z $-valued. If $ Q $ is the quadratic form associated to $ G $ and $ M_Q $ is the matrix of $ Q $, then $ \rank(M_Q \bmod p) \leq n $.
\end{proposition}
\begin{proof}
We introduce the following notation: if $ M $ is a $ 2n\times 2n $ matrix and $ 1 \leq k \leq 2n $, let $ [M]_k $ denote the $ k $-th row of $ M $.

If $ 1 \leq i \leq j \leq n $, then we let $ \alpha_{ij} = r_{ij}+s_{ij}\omega_D $ and $ \beta_{ij} = t_{ij}+u_{ij}\omega_D $ where $ r_{ij}, s_{ij}, t_{ij}, u_{ij} \in \Q $.

First assume that $ D \equiv 2, 3 \pmod{4} $. If $ 1 \leq k \leq n $, then the $2k$-th row of $ M_Q $ equals
\[
    [M_Q]_{2k} = \begin{pmatrix}
        \frac{b_{1k}}{2}&\frac{c_{1k}}{2}&\dots&\frac{b_{kk}}{2}&c_{kk}&\dots&\frac{b_{nk}}{2}&\frac{c_{kn}}{2}
    \end{pmatrix}.
\]
By \Cref{lemmaMQn}, if $ 1 \leq i < j \leq n $, then
\[
    \frac{b_{ij}}{2} = D(s_{ij}-u_{ij}),\ \frac{b_{ji}}{2} = D(s_{ij}+u_{ij}),\ \frac{c_{ij}}{2} = D(r_{ij}-t_{ij}),
\]
and if $ 1 \leq i \leq n $, then
\[
    \frac{b_{ii}}{2} = 2Ds_{ii},\ c_{ii} = D(2r_{ii}-t_{ii}).
\]
Since we assume that $ G $ is integral, $ r_{ij}, s_{ij}, t_{ij}, u_{ij} \in \Z $. Thus $ D $ divides every element of $ [M_Q]_{2k} $. It follows that $ M_Q \bmod{p} $ contains $ n $ zero rows, hence $ \rank(M_Q \bmod{p}) \leq n $.

Secondly, assume that $ D \equiv 1 \pmod{4} $. Let $ N_Q $ be the matrix obtained from $ M_Q $ by subtracting two times the $ 2k $-th row of $ M_Q $ from the $ (2k-1) $-th row of $ M_Q $ for every $ 1 \leq k \leq n $. The $ (2k-1) $-th row of $ N_Q $ equals
\[
    [M_Q]_{2k-1}-2[M_Q]_{2k} = \begin{pmatrix}
        \frac{a_{1k}}{2}-b_{1k}&\frac{b_{k1}}{2}-c_{1k}&\dots&a_{kk}-b_{kk}&\frac{b_{kk}}{2}-2c_{kk}&\dots&\frac{a_{kn}}{2}-b_{nk}&\frac{b_{kn}}{2}-c_{kn}
    \end{pmatrix}.
\]
By \Cref{lemmaMQn}, if $ 1 \leq i < j \leq n $, then
\begin{align*}
    \frac{a_{ij}}{2}-b_{ij}& = \frac{1}{2}\left(2r_{ij}+s_{ij}+2t_{ij}+u_{ij}\right)-\left(r_{ij}+\frac{1+D}{2}s_{ij}+t_{ij}+\frac{1-D}{2}u_{ij}\right) = D\cdot\frac{u_{ij}+s_{ij}}{2},\\
    \frac{b_{ji}}{2}-c_{ij}& = \frac{1}{2}\left(r_{ij}+\frac{1+D}{2}s_{ij}+t_{ij}+\frac{1+D}{2}u_{ij}\right)-\left(\frac{1+D}{2}r_{ij}+\frac{1+3D}{4}s_{ij}+\frac{1-D}{2}t_{ij}+\frac{1-D}{4}u_{ij}\right)\\
    & = D\cdot\frac{t_{ij}+u_{ij}-r_{ij}-s_{ij}}{2},
\end{align*}
and if $ 1 \leq i \leq n $, then
\begin{align*}
    a_{ii}-b_{ii}& = 2r_{ii}+s_{ii}+t_{ii}-\left(2r_{ii}+(1+D)s_{ii}+t_{ii}\right) = -Ds_{ii},\\
    \frac{b_{ii}}{2}-2c_{ii}& = \frac{1}{2}\left(2r_{ii}+(1+D)s_{ii}+t_{ii}\right)-2\left(\frac{1+D}{2}r_{ii}+\frac{1+3D}{4}s_{ii}+\frac{1-D}{4}t_{ii}\right) = -Dr_{ii}-Ds_{ii}+\frac{D}{2}t_{ii}.
\end{align*}
We have $ r_{ij}, s_{ij}, t_{ij}, u_{ij} \in \frac{1}{2}\Z $ because $ G $ is integral. Since we assume $ D \equiv 1 \pmod{4} $, $ p $ is an odd prime and $ 2 $ is invertible modulo $ p $. We see that every element of $ [N_Q]_{2k-1} $ vanishes modulo $ p $, hence $ N_Q \bmod{p} $ contains $ n $ zero rows and $ \rank(M_Q \bmod{p}) = \rank(N_Q \bmod{p}) \leq n $.
\end{proof}

\section{Binary and ternary generalized quadratic forms}
\label{secBinary}

Let us formally state a definition from the introduction.

\begin{definition}
\label{defZUniv}
Let $ K = \Q(\sqrt{D}) $ where $ D \in \Z_{\geq 2} $ is squarefree and let $ G $ be an integral generalized form in $ n $ variables over $ K $. Assume that $ G $ is $ \Z $-valued. If for every $ a \in \Z_{\geq 1} $, there exist $ \alpha_1, \alpha_2, \dots, \alpha_n \in \O_K $ such that
\[
    G(\alpha_1, \alpha_2, \dots, \alpha_n) = a,
\]
then we say that $ G $ is \emph{$ \Z $-universal}.
\end{definition}

The purpose of this section is to investigate binary and ternary $ \Z $-universal forms. We begin by proving one implication in statements (i) and (ii) of \Cref{thmBinary}.
\begin{proposition}
\label{propBinary}
Let $ K = \Q(\sqrt{D}) $ where $ D \in \Z_{\geq 2} $ is squarefree.
\begin{enumerate}[(i)]
    \item If $ D \equiv 2,3 \pmod{4} $ and there exists a positive definite binary generalized form over $ K $ which is $ \Z $-universal, then $ D \in \{2,3,6,7,10\} $.
    \item If $ D \equiv 1 \pmod{4} $ and there exists a \emph{classical} positive definite binary generalized form over $ K $ which is $ \Z $-universal, then $ D = 5 $.
\end{enumerate}
\end{proposition}
\begin{proof}
Suppose that $ G $ is a positive definite binary generalized form over $ K $ which is $ \Z $-universal. If $ D \equiv 1 \pmod{4} $, we further assume that $ G $ is classical. If $ Q $ is the associated quaternary quadratic form defined by \eqref{eqDefQ}, then $ Q $ is positive definite and universal.

By \Cref{lemmaGQ}, $ Q $ is also classical. Bhargava obtained, as a corollary of his proof of the 15-Theorem, a complete list of classical quaternary positive definite universal quadratic forms up to equivalence \cite[Table 5]{Bh}. Every form on this list has determinant $ \leq 112 $. By \Cref{corDetMQ}, $ D^2 \mid \det(M_Q) $, thus $ D^2 \leq 112 $. This yields the proof.
\end{proof}

\begin{lemma}
\label{lemmaMQBin}
Let $ K = \Q(\sqrt{D}) $ where $ D \in \Z\setminus\{0,1\} $ is squarefree and let $ G $ be a binary generalized form over $ K $ given by the matrix
\begin{equation}
\label{eqMGBinary}
    M_G = \begin{pmatrix}
        a&\frac{b}{2}&\frac{d}{2}&\frac{e}{2}\\
        \frac{b}{2}&c&\frac{\tau(e)}{2}&\frac{f}{2}\\
        \frac{d}{2}&\frac{\tau(e)}{2}&\tau(a)&\frac{\tau(b)}{2}\\
        \frac{e}{2}&\frac{f}{2}&\frac{\tau(b)}{2}&\tau(c)
    \end{pmatrix}
\end{equation}
where $ a, b, c, e \in K $ and $ d, f \in \Q $. Let $ a = a_1+a_2\omega_D $, $ b = b_1+b_2\omega_D $, $ c = c_1+c_2\omega_D $, and $ e = e_1+e_2\omega_D $. Let $ Q $ be the quaternary quadratic form associated to $ G $ and $ M_Q $ be the matrix of $ Q $.

If $ D \equiv 2, 3 \pmod{4} $, then
\[
    M_Q = \begin{pmatrix}
        2a_1+d&2Da_2&b_1+e_1&D(b_2-e_2)\\
        2Da_2&D(2a_1-d)&D(b_2+e_2)&D(b_1-e_1)\\
        b_1+e_1&D(b_2+e_2)&2c_1+f&2Dc_2\\
        D(b_2-e_2)&D(b_1-e_1)&2Dc_2&D(2c_1-f)
    \end{pmatrix}.
\]

If $ D \equiv 1 \pmod{4} $, then
\[
    M_Q = \begin{pmatrix}
        M_{Q,1}&M_{Q,2}\\
        M_{Q,2}^\top&M_{Q,3}
    \end{pmatrix}
\]
where
\begin{align*}
    M_{Q,1}& = \begin{pmatrix}
        2a_1+a_2+d&a_1+\frac{1+D}{2}a_2+\frac{d}{2}\\
        a_1+\frac{1+D}{2}a_2+\frac{d}{2}&\frac{1+D}{2}a_1+\frac{1+3D}{4}a_2+\frac{1-D}{4}d
    \end{pmatrix},\\
    M_{Q,2}& = \begin{pmatrix}
        b_1+\frac{b_2}{2}+e_1+\frac{e_2}{2}&\frac{b_1}{2}+\frac{1+D}{4}b_2+\frac{e_1}{2}+\frac{1-D}{4}e_2\\
        \frac{b_1}{2}+\frac{1+D}{4}b_2+\frac{e_1}{2}+\frac{1+D}{4}e_2&\frac{1+D}{4}b_1+\frac{1+3D}{8}b_2+\frac{1-D}{4}e_1+\frac{1-D}{8}e_2
    \end{pmatrix},\\
    M_{Q,3}& = \begin{pmatrix}
        2c_1+c_2+f&c_1+\frac{1+D}{2}c_2+\frac{f}{2}\\
        c_1+\frac{1+D}{2}c_2+\frac{f}{2}&\frac{1+D}{2}c_1+\frac{1+3D}{4}c_2+\frac{1-D}{4}f
    \end{pmatrix}.
\end{align*}
\end{lemma}
\begin{proof}
The matrix of $ G $ is
\[
    M_G = \begin{pmatrix}
        \alpha_{11}&\frac{\alpha_{12}}{2}&\frac{\beta_{11}}{2}&\frac{\beta_{12}}{2}\\
        \frac{\alpha_{12}}{2}&\alpha_{22}&\frac{\tau(\beta_{12})}{2}&\frac{\beta_{22}}{2}\\
        \frac{\beta_{11}}{2}&\frac{\tau(\beta_{12})}{2}&\tau(\alpha_{11})&\frac{\tau(\alpha_{12})}{2}\\
        \frac{\beta_{12}}{2}&\frac{\beta_{22}}{2}&\frac{\tau(\alpha_{12})}{2}&\tau(\alpha_{22})
    \end{pmatrix}
\]
where $ \alpha_{11} = a $, $ \alpha_{12} = b $, $ \alpha_{22} = c $, $ \beta_{11} = d $, $ \beta_{12} = e $, and $ \beta_{22} = f $. In the notation of \Cref{lemmaMQn}, $ \alpha_{ij} = r_{ij}+s_{ij}\omega_D $ and $ \beta_{ij} = t_{ij}+u_{ij}\omega_D $, hence $ r_{11} = a_1 $, $ s_{11} = a_2 $, $ r_{12} = b_1 $, $ s_{12} = b_2 $, $ r_{22} = c_1 $, $ s_{22} = c_2 $, $ t_{11} = d $, $ t_{12} = e_1 $, $ u_{12} = e_2 $, and $ t_{22} = f $.

The matrix of $ Q $ is
\[
    M_Q = \begin{pmatrix}
        a_{11}&\frac{b_{11}}{2}&\frac{a_{12}}{2}&\frac{b_{12}}{2}\\
        \frac{b_{11}}{2}&c_{11}&\frac{b_{21}}{2}&\frac{c_{12}}{2}\\
        \frac{a_{12}}{2}&\frac{b_{21}}{2}&a_{22}&\frac{b_{22}}{2}\\
        \frac{b_{12}}{2}&\frac{c_{12}}{2}&\frac{b_{22}}{2}&c_{22}
    \end{pmatrix}.
\]

If $ D \equiv 2,3 \pmod{4} $, then by \Cref{lemmaMQn},
\begin{align*}
    a_{11}& = 2r_{11}+t_{11} = 2a_1+d,\\
    b_{11}& = 4Ds_{11} = 4Da_2,\\
    c_{11}& = 2Dr_{11}-Dt_{11} = D(2a_1-d),\\
    a_{12}& = 2r_{12}+2t_{12} = 2(b_1+e_1),\\
    b_{12}& = 2Ds_{12}-2Du_{12} = 2D(b_2-e_2),\\
    b_{21}& = 2Ds_{12}+2Du_{12} = 2D(b_2+e_2),\\
    c_{12}& = 2Dr_{12}-2Dt_{12} = 2D(b_1-e_1),\\
    a_{22}& = 2r_{22}+t_{22} = 2c_1+f,\\
    b_{22}& = 4Ds_{22} = 4Dc_2,\\
    c_{22}& = 2Dr_{22}-Dt_{22} = D(2c_1-f).
\end{align*}

If $ D \equiv 1 \pmod{4} $, then by \Cref{lemmaMQn},
\begin{align*}
    a_{11}& = 2r_{11}+s_{11}+t_{11} = 2a_1+a_2+d,\\
    b_{11}& = 2r_{11}+(1+D)s_{11}+t_{11} = 2a_1+(1+D)a_2+d,\\
    c_{11}& = \frac{1+D}{2}r_{11}+\frac{1+3D}{4}s_{11}+\frac{1-D}{4}t_{11} = \frac{1+D}{2}a_1+\frac{1+3D}{4}a_2+\frac{1-D}{4}d,\\
    a_{12}& = 2r_{12}+s_{12}+2t_{12}+u_{12} = 2b_1+b_2+2e_1+e_2,\\
    b_{12}& = r_{12}+\frac{1+D}{2}s_{12}+t_{12}+\frac{1-D}{2}u_{12} = b_1+\frac{1+D}{2}b_2+e_1+\frac{1-D}{2}e_2,\\
    b_{21}& = r_{12}+\frac{1+D}{2}s_{12}+t_{12}+\frac{1+D}{2}u_{12} = b_1+\frac{1+D}{2}b_2+e_1+\frac{1+D}{2}e_2,\\
    c_{12}& = \frac{1+D}{2}r_{12}+\frac{1+3D}{4}s_{12}+\frac{1-D}{2}t_{12}+\frac{1-D}{4}u_{12} = \frac{1+D}{2}b_1+\frac{1+3D}{4}b_2+\frac{1-D}{2}e_1+\frac{1-D}{4}e_2,\\
    a_{22}& = 2r_{22}+s_{22}+t_{22} = 2c_1+c_2+f,\\
    b_{22}& = 2r_{22}+(1+D)s_{22}+t_{22} = 2c_1+(1+D)c_2+f,\\
    c_{22}& = \frac{1+D}{2}r_{22}+\frac{1+3D}{4}s_{22}+\frac{1-D}{4}t_{22} = \frac{1+D}{2}c_1+\frac{1+3D}{4}c_2+\frac{1-D}{4}f.\qedhere
\end{align*}
\end{proof}

    In \Crefrange{lemmaExD2}{lemmaExD10}, we find an example of a $ \Z $-universal generalized form over $ \Q(\sqrt{D}) $ for $ D \in \{2,3,6,7,10\} $ and of a classical $ \Z $-universal generalized form over $ \Q(\sqrt{5}) $. We prove only \Cref{lemmaExD2}, the rest is stated and proved in the Appendix in \Cref{SectionFillingInPRoofs55-59}.
\begin{lemma}
\label{lemmaExD2}
Let $ K = \Q(\sqrt{2}) $ and
\[
    G_2(z, w) = z^2-\sqrt{2}zw+w^2-z\tau(z)-w\tau(w)+\tau(z)^2+\sqrt{2}\tau(z)\tau(w)+\tau(w)^2.
\]
The generalized form $ G_2 $ is $ \Z $-universal over $ K $.
\end{lemma}
\begin{proof}
The coefficients of $ G_2 $ are $ a = 1 $, $ b = -\sqrt{2} $, $ c = 1 $, $ d = -1 $, $ e = 0 $, $ f = -1 $. Let $ Q_2 $ be the quaternary quadratic form $ Q_2(x_1, y_1, x_2, y_2) = G_2(x_1+y_1\sqrt{2}, x_2+y_2\sqrt{2}) $ and let $ M_{Q_2} $ be the matrix of $ Q_2 $. By \Cref{lemmaMQBin} with $ D = 2 $,
\[
     M_{Q_2} = \begin{pmatrix}
        2a_1+d&4a_2&b_1+e_1&2(b_2-e_2)\\
        4a_2&2(2a_1-d)&2(b_2+e_2)&2(b_1-e_1)\\
        b_1+e_1&2(b_2+e_2)&2c_1+f&4c_2\\
        2(b_2-e_2)&2(b_1-e_1)&4c_2&2(2c_1-f)
     \end{pmatrix} = \begin{pmatrix}
        1&0&0&-2\\
        0&6&-2&0\\
        0&-2&1&0\\
        -2&0&0&6
    \end{pmatrix}.
\]
We get
\[
    M_{Q_2} \sim \begin{pmatrix}
        1&0&0&0\\
        0&6&-2&0\\
        0&-2&1&0\\
        0&0&0&2
    \end{pmatrix} \sim \begin{pmatrix}
        1&0&0&0\\
        0&2&0&0\\
        0&0&1&0\\
        0&0&0&2
    \end{pmatrix}.
\]
Since the last matrix is positive definite, $ G_2 $ is also positive definite. Moreover, the diagonal quadratic form $ x^2+2y^2+z^2+2w^2 $ is universal by the 15-Theorem, hence $ G_2 $ is $ \Z $-universal.
\end{proof}

\begin{lemma}
\label{lemmaExD3}
Let $ K = \Q(\sqrt{3}) $ and
\[
    G_3(z, w) = 2z^2-2\sqrt{3}zw+2w^2-3z\tau(z)-\sqrt{3}z\tau(w)+\sqrt{3}w\tau(z)+2\tau(z)^2+2\sqrt{3}\tau(z)\tau(w)+2\tau(w)^2.
\]
The generalized form $ G_3 $ is $ \Z $-universal over $ K $.
\end{lemma}

\begin{lemma}
\label{lemmaExD5}
Let $ K = \Q(\sqrt{5}) $ and
\[
    G_5(z, w) = (3-\omega_5)z^2+w^2-4z\tau(z)+(3-\tau(\omega_5))\tau(z)^2+\tau(w)^2
\]
where $ \omega_5 = \frac{1+\sqrt{5}}{2} $. The generalized form $ G_5 $ is classical and $ \Z $-universal over $ K $.
\end{lemma}

\begin{lemma}
\label{lemmaExD6}
Let $ K = \Q(\sqrt{6}) $ and $ G_6 $ be the generalized form with the matrix
\[
    M_{G_6} = \begin{pmatrix}
        3&\sqrt{6}&-\frac{5}{2}&\frac{\sqrt{6}}{2}\\
        \sqrt{6}&3&-\frac{\sqrt{6}}{2}&-\frac{1}{2}\\
        -\frac{5}{2}&-\frac{\sqrt{6}}{2}&3&-\sqrt{6}\\
        \frac{\sqrt{6}}{2}&-\frac{1}{2}&-\sqrt{6}&3
    \end{pmatrix}.
\]
The generalized form $ G_6 $ is $ \Z $-universal over $ K $.
\end{lemma}

\begin{lemma}
\label{lemmaExD7}
Let $ K = \Q(\sqrt{7}) $ and $ G_7 $ be the generalized quadratic form with the matrix
\[
    M_{G_7} = \begin{pmatrix}
        3&\frac{1+2\sqrt{7}}{2}&-\frac{5}{2}&\frac{-1+2\sqrt{7}}{2}\\
        \frac{1+2\sqrt{7}}{2}&4&\frac{-1-2\sqrt{7}}{2}&\frac{3}{2}\\
        -\frac{5}{2}&\frac{-1-2\sqrt{7}}{2}&3&\frac{1-2\sqrt{7}}{2}\\
        \frac{-1+2\sqrt{7}}{2}&\frac{3}{2}&\frac{1-2\sqrt{7}}{2}&4
    \end{pmatrix}.
\]
The generalized form $ G_7 $ is $ \Z $-universal over $ K $.
\end{lemma}

\begin{lemma}
\label{lemmaExD10}
Let $ K = \Q(\sqrt{10}) $. If $ G_{10} $ is the generalized quadratic form with the matrix
\[
    M_{G_{10}} = \begin{pmatrix}
        3&\frac{\sqrt{10}}{2}&-\frac{5}{2}&\frac{\sqrt{10}}{2}\\
        \frac{\sqrt{10}}{2}&5&-\frac{\sqrt{10}}{2}&\frac{7}{2}\\
        -\frac{5}{2}&-\frac{\sqrt{10}}{2}&3&\frac{-\sqrt{10}}{2}\\
        \frac{\sqrt{10}}{2}&\frac{7}{2}&\frac{\sqrt{10}}{2}&5
    \end{pmatrix},
\]
then $ G_{10} $ is $ \Z $-universal over $ K $.
\end{lemma}

Now we are ready to prove \Cref{thmBinary}, which we restate here for the convenience of the reader.

\begin{theorem}
\label{thmBinary2}
Let $ K = \Q(\sqrt{D}) $ where $ D \in \Z_{\geq 2} $ is squarefree.
\begin{enumerate}[i)]
\item Assume that $ D \equiv 2, 3 \pmod{4} $. A binary $ \Z $-universal generalized quadratic form exists over $ K $ if and only if $ D \in \{2,3,6,7,10\} $.
\item Assume that $ D \equiv 1 \pmod{4} $. A \emph{classical} binary $ \Z $-universal generalized quadratic form exists over $ K $ if and only if $ D = 5 $.
\end{enumerate}
\end{theorem}
\begin{proof}
First we prove (i). We assume that $ D \equiv 2, 3 \pmod{4} $. If a binary $ \Z $-universal generalized form exists over $ K = \Q(\sqrt{D}) $, then $ D \in \{2,3,6,7,10\} $ by \Cref{propBinary} (i). Conversely, in \Cref{lemmaExD2,lemmaExD3,lemmaExD6,lemmaExD7,lemmaExD10} we found an example of a generalized form with the required properties over each of these fields.

Next, we prove (ii). We assume that $ D \equiv 1 \pmod{4} $. If a classical binary $ \Z $-universal generalized form exists over $ K = \Q(\sqrt{D}) $, then $ D = 5 $ by \Cref{propBinary} (ii). We found an example of such a form over $ \Q(\sqrt{5}) $ in \Cref{lemmaExD5}.
\end{proof}

\Cref{thmTernary} is an immediate consequence of the following slightly refined version.
\begin{theorem}
\label{thmTernary2}
Let $ K = \Q(\sqrt{D}) $ where $ D \in \Z_{\geq 2} $ is squarefree. Assume that one of the following two conditions is satisfied:
\begin{enumerate}[i)]
\item $ D \equiv 2, 3 \pmod{4} $ and there exists a ternary $ \Z $-universal generalized quadratic form over $ K $,
\item $ D \equiv 1 \pmod{4} $ and there exists a \emph{classical} ternary $ \Z $-universal generalized quadratic form over $ K $.
\end{enumerate}
Then $ D \in \{1,\dots,62\}\cup\{66,70,74,77,78,82,85,86,87,93,94,95,102,106,110\} $.
\end{theorem}
\begin{proof}
Let $ G $ be a ternary generalized quadratic form over $ K $ which is $ \Z $-universal. If $ D \equiv 1 \pmod{4} $, then we further assume that $ G $ is classical. Let $ Q $ be the quadratic form in $ 6 $ variables associated to $ G $ and let $ M_Q $ be the matrix of $ Q $. By \Cref{lemmaGQ}, $ Q $ is classical.

A classical universal quadratic form has a quaternary subform (not necessarily universal) obtained by the method of escalations \cite[p. 31]{Bh}. Let $ R $ be such a subform of $ Q $ and $ M_R $ the matrix of $ R $. If $ p \mid D $ is a prime, then $ \rank(M_Q \bmod{p}) \leq 3 $ by \Cref{propRank}. The rank of $ R $ over $ \Z/p\Z $ is less than or equal to the rank of $ Q $ over $ \Z/p\Z $, hence $ \rank(M_R \bmod{p}) \leq 3 $. Since $ M_R $ is a $ 4 \times 4 $ matrix, $ \det(M_R) \equiv 0 \pmod{p} $.

This proves $ D \mid \det(M_R) $. Every quaternary form obtained by the method of escalations is listed in \cite[Table 3]{Bh} and in particular, has determinant $ \leq 112 $. It remains to determine which squarefree $ D \leq 112 $ divide the determinant of some form on the list. Every squarefree $ D \in \{1, \dots, 62\} $ satisfies this condition. Additionally, the squarefree determinants greater than $ 62 $ on the list are $ 66,70,74,77,78,82,85,86,87,93,94,95,102,106,110 $.
\end{proof}

Using the 290-Theorem, we can remove the ``classical" assumption in the case $ D \equiv 1 \pmod{4} $.

\begin{theorem}
\label{thmBinTer2}
There are only finitely many real quadratic fields which admit a binary or ternary $ \Z $-universal generalized quadratic form.
\end{theorem}
\begin{proof}
We show that there exists a positive constant $ C $ such that $ D \leq C $ for every such field $ K = \Q(\sqrt{D}) $. If $ D \equiv 2,3 \pmod{4} $, then the statement follows from \Cref{thmBinary2,thmTernary2}. Hence, we can assume that $ D \equiv 1 \pmod{4} $. Let $ G $ be a $ \Z $-universal generalized form over $ K $ with matrix $ M_G $ and $ Q $ be the quadratic form associated to $ G $ with matrix $ M_Q $. 

First, assume that $ G $ is binary. If $ M_G $ is the matrix of $ G $ and $ M_Q $ is the matrix of $ Q $, then $ \det(M_Q) = D^2\det(M_G) $ by \Cref{lemmaDetMQ}. The matrix $ M_G $ is a $ 4\times 4 $ matrix with entries in $ \frac{1}{2}\O_K $, hence $ 16\det(M_G) \in \O_K \cap \Q = \Z $, and we get $ D^2 \mid 16\det(M_Q) $.

A corollary of the 290-Theorem is that there are exactly 6436 universal quaternary forms \cite[Theorem 4]{BH}. If we let $ m $ be the maximum of their determinants, then $ D^2 \leq 16m $.

Secondly, assume that $ G $ is ternary. A universal quadratic form contains a quaternary subform equivalent to one of $ 6560 $ basic escalators \cite[p. 4]{BH}. Let $ R $ be such a subform of $ Q $ and $ M_R $ the matrix of $ R $. If $ p \mid D $ is a prime, then $ \rank(M_Q \bmod{p}) \leq 3 $ by \Cref{propRank}, hence $ \rank(M_R \bmod{p}) \leq 3 $ as well. Since $ M_R $ is a $ 4\times 4 $ matrix, $ \det(M_R) \equiv 0 \pmod{p} $, and since the entries of $ M_R $ are in $ \frac{1}{2}\Z $, we have $ 16\det(M_R) \in \Z $.

This proves $ D \mid 16\det(M_R) $. If we let $ m $ be the maximum of the determinants of the basic escalators, then $ D \leq 16m $.
\end{proof}

In contrast to \Cref{thmBinTer2}, it is not difficult to construct a $ \Z $-universal generalized form in $ 4 $ variables over every real quadratic field.

\begin{proposition}
\label{propFourVars}
Let $ K = \Q(\sqrt{D}) $ where $ D \in \Z_{\geq 2} $ is squarefree. The generalized quadratic form
\begin{align*}
    G(z_1,z_2,z_3,z_4)& = z_1^2-z_1\tau(z_1)+\tau(z_1)^2+z_2^2-z_2\tau(z_2)+\tau(z_2)^2\\
    &+z_3^2-z_3\tau(z_3)+\tau(z_3)^2+z_4^2-z_4\tau(z_4)+\tau(z_4)^2
\end{align*}
is $ \Z $-universal over $ K $. Thus, the minimal rank of a $ \Z $-universal generalized form over $ K $ is at most $ 4 $.
\end{proposition}
\begin{proof}
Clearly, $ G $ is positive definite. It is enough to show that $ G $ represents every $ a \in \Z_{\geq 1} $ over $ \Z $. If $ z_i \in \Z $ for $ i = 1, \dots, 4 $, then $ \tau(z_i) = z_i $, hence
\[
    G(z_1, z_2, z_3, z_4) = z_1^2+z_2^2+z_3^2+z_4^2.
\]
This quadratic form is universal by Lagrange's Four Square Theorem.
\end{proof}

We end this section with a proposition that is used in \Cref{secExamples2}.

\begin{proposition}
\label{propDetBinary}
Let $ K = \Q(\sqrt{D}) $ where $ D \in \Z\setminus\{0, 1\} $ is squarefree, $ D \equiv 2, 3 \pmod{4} $. Let $ G $ be an integral binary generalized form over $ K $. Assume that $ G $ is $ \Z $-valued. If $ Q $ is the quaternary quadratic form associated to $ G $ and $ M_Q $ is the matrix of $ Q $, then
\[
    \frac{\det(M_Q)}{D^2} \equiv 0, 1 \pmod{4}.
\]
\end{proposition}
\begin{proof}
By \Cref{lemmaDetMQ},
\[
    \det(M_Q) = D^2\cdot\det(2M_G).
\]
The matrix $ 2M_G $ is
\[
    2M_G = \begin{pmatrix}
        2a&b&d&e\\
        b&2c&\tau(e)&f\\
        d&\tau(e)&2\tau(a)&\tau(b)\\
        e&f&\tau(b)&2\tau(c)
    \end{pmatrix}
\]
where $ a, b, c, e \in \O_K $ and $ d, f \in \Z $ because $ G $ is integral. Expanding $ \det(2M_G) $ along the first column, we get
\begin{align*}
    \det(2M_G) =& 2a\cdot \det\begin{pmatrix}
        2c&\tau(e)&f\\
        \tau(e)&2\tau(a)&\tau(b)\\
        f&\tau(b)&2\tau(c)
    \end{pmatrix}-b\cdot\det\begin{pmatrix}
        b&d&e\\
        \tau(e)&2\tau(a)&\tau(b)\\
        f&\tau(b)&2\tau(c)
    \end{pmatrix}+\\
    &+d\cdot\det\begin{pmatrix}
        b&d&e\\
        2c&\tau(e)&f\\
        f&\tau(b)&2\tau(c)
    \end{pmatrix}-e\cdot\det\begin{pmatrix}
        b&d&e\\
        2c&\tau(e)&f\\
        \tau(e)&2\tau(a)&\tau(b)
    \end{pmatrix},
\end{align*}
hence
\begin{align*}
    \det(2M_G) \equiv & 2a\Big(\tau(e)\tau(b)f+f\tau(e)\tau(b)\Big)-b\Big(d\tau(b)f+e\tau(e)\tau(b)-2f\tau(a)e-\tau(b)^2b-2\tau(c)\tau(e)d\Big)\\
    &+d\Big(2b\tau(e)\tau(c)+df^2+2ec\tau(b)-f\tau(e)e-\tau(b)fb\Big)\\
    &-e\Big(b\tau(e)\tau(b)+df\tau(e)-\tau(e)^2e-2\tau(a)fb-2\tau(b)cd\Big)\\
    \equiv& -\Nm(b)df-\Nm(b)\Nm(e)+2\tau(a)bef+\Nm(b)^2+2b\tau(c)d\tau(e)\\
    &+2b\tau(c)d\tau(e)+d^2f^2+2\tau(b)cde-d\Nm(e)f-\Nm(b)df\\
    &-\Nm(b)\Nm(e)-d\Nm(e)f+\Nm(e)^2+2\tau(a)bef+2\tau(b)cde\\
    \equiv &-2\Nm(b)df-2\Nm(b)\Nm(e)+\Nm(b)^2+d^2f^2-2d\Nm(e)f+\Nm(e)^2\\
    \equiv & \left(\Nm(b)+\Nm(e)+df\right)^2 \pmod{4}.
\end{align*}
Thus,
\[
    \frac{\det(M_Q)}{D^2} = \det(2M_G) \equiv 0, 1 \pmod{4}.\qedhere
\]
\end{proof}

\section{Indefinite generalized quadratic forms}
\label{secIndef}

We end with a short section providing a shortened proof of the universality of $ G(z, w) = z\tau(z) +w\tau(w) $ over real quadratic fields, where the form is indefinite. For a more detailed approach, please consult \Cref{secIndef2} of the Appendix.

\begin{theorem}
\label{thmIndef}
Let $ K = \Q(\sqrt{D}) $ where $ D \in \Z_{\geq 2} $ is squarefree. The binary generalized form
\[
    G(z, w) = z\tau(z)+w\tau(w)
\]
represents every $ a \in \Z $.
\end{theorem}
\begin{proof}
It is sufficient to show that $ G $ represents every $ a \in \Z\setminus\{0\} $ over $ \Z[\sqrt{D}] $ because $ \Z[\sqrt{D}] \subset \O_K $.

Let $ z = x_1+y_1\sqrt{D} $ and $ w = x_2+y_2\sqrt{D} $ where $ x_1, y_1, x_2, y_2 \in \Z $. Let
\[
    Q(x_1, y_1, x_2, y_2) = G(x_1+y_1\sqrt{D}, x_2+y_2\sqrt{D}).
\]
We have
\[
    Q(x_1, y_1, x_2, y_2) = (x_1+y_1\sqrt{D})(x_1-y_1\sqrt{D})+(x_2+y_2\sqrt{D})(x_2-y_2\sqrt{D}) = x_1^2+x_2^2-Dy_1^2-Dy_2^2.
\]

By the local-global principle for representations over $\Z$ by indefinite forms (see, e.g., \cite[Chap. 9, Thm. 1.5]{Ca}), it suffices to verify that each $a\in\Z$ is represented over $\Z_p$ for all primes $p$. This is indeed straightforward to verify; for details see Section \ref{secIndef2} of the Appendix.\end{proof}

It is natural to be interested in the ways a quadratic form represents 0 non-trivially. In Proposition \ref{propRepZero}, we show that the quadratic form $Q(x,y,z,w)=x^2+y^2-Dz^2-Dw^2$ represents 0 non-trivially if and only if $D$ is a sum of two squares. The proof of this proposition can be found in the appendix.

\newcommand{\sectionletter}{A}
\renewcommand{\thesection}{\sectionletter\arabic{section}}
\newpage
\renewcommand{\sectionletter}{A}

\setcounter{section}{3}
\section{Proofs of \Crefrange{lemmaExD3}{lemmaExD10}}
\label{SectionFillingInPRoofs55-59}

\setcounter{theorem}{4} 
\begin{lemma}
Let $ K = \Q(\sqrt{3}) $ and
\[
    G_3(z, w) = 2z^2-2\sqrt{3}zw+2w^2-3z\tau(z)-\sqrt{3}z\tau(w)+\sqrt{3}w\tau(z)+2\tau(z)^2+2\sqrt{3}\tau(z)\tau(w)+2\tau(w)^2.
\]
The generalized form $ G_3 $ is $ \Z $-universal over $ K $.
\end{lemma}
\begin{proof}
The coefficients of $ G_3 $ are $ a = 2 $, $ b = -2\sqrt{3} $, $ c = 2 $, $ d = -3 $, $ e = -\sqrt{3} $, $ f = 0 $. Let $ Q_3 $ be the quaternary quadratic form $ Q_3(x_1, y_1, x_2, y_2) = G_3(x_1+y_1\sqrt{3}, x_2+y_2\sqrt{3}) $. By \Cref{lemmaMQBin} with $ D = 3 $,
\[
    M_{Q_3} = \begin{pmatrix}
        2a_1+d&6a_2&b_1+e_1&3(b_2-e_2)\\
        6a_2&3(2a_1-d)&3(b_2+e_2)&3(b_1-e_1)\\
        b_1+e_1&3(b_2-e_2)&2c_1+f&6c_2\\
        3(b_2-e_2)&3(b_1-e_1)&6c_2&3(2c_1-f)
    \end{pmatrix} = \begin{pmatrix}
        1&0&0&-3\\
        0&21&-9&0\\
        0&-9&4&0\\
        -3&0&0&12
    \end{pmatrix}.
\]
We get
\[
    M_{Q_3} \sim \begin{pmatrix}
        1&0&0&0\\
        0&21&-9&0\\
        0&-9&4&0\\
        0&0&0&3
    \end{pmatrix} \sim \begin{pmatrix}
        1&0&0&0\\
        0&1&-1&0\\
        0&-1&4&0\\
        0&0&0&3
    \end{pmatrix} \sim \begin{pmatrix}
        1&0&0&0\\
        0&1&0&0\\
        0&0&3&0\\
        0&0&0&3
    \end{pmatrix},
\]
The diagonal quadratic form $ x^2+y^2+3z^2+3w^2 $ is positive definite and universal by the 15-Theorem, hence $ G_3 $ is also positive definite and $ \Z $-universal.
\end{proof}

\begin{lemma}

Let $ K = \Q(\sqrt{5}) $ and
\[
    G_5(z, w) = (3-\omega_5)z^2+w^2-4z\tau(z)+(3-\tau(\omega_5))\tau(z)^2+\tau(w)^2
\]
where $ \omega_5 = \frac{1+\sqrt{5}}{2} $. The generalized form $ G_5 $ is classical and $ \Z $-universal over $ K $.
\end{lemma}
\begin{proof}
The coefficients of $ G_5 $ are $ a = 3-\omega_5 $, $ b = 0 $, $ c = 1 $, $ d = -4 $, $ e = 0 $, and $ f = 0 $. Since the coefficients of the off-diagonal terms are even, $ G_5 $ is classical. Let $ Q_5 $ be the quaternary quadratic form $ Q_5(x_1, y_1, x_2, y_2) = G_5(x_1+y_1\omega_5, x_2+y_2\omega_5) $. By \Cref{lemmaMQBin} with $ D = 5 $,
\[
    M_{Q_5} = \begin{pmatrix}
        M_{Q_5, 1}&M_{Q_5, 2}\\
        M_{Q_5, 2}^\top&M_{Q_5, 3}
    \end{pmatrix}
\]
where
\begin{align*}
    M_{Q_5, 1}& = \begin{pmatrix}
        2a_1+a_2+d&a_1+3a_2+\frac{d}{2}\\
        a_1+3a_2+\frac{d}{2}&3a_1+4a_2-d
    \end{pmatrix} = \begin{pmatrix}
        1&-2\\
        -2&9
    \end{pmatrix},\\
    M_{Q_5, 2}& = \begin{pmatrix}
        b_1+\frac{b_2}{2}+e_1+\frac{e_2}{2}&\frac{b_1}{2}+\frac{3}{2}b_2+\frac{e_1}{2}-e_2\\
        \frac{b_1}{2}+\frac{3}{2}b_2+\frac{e_1}{2}+\frac{3}{2}e_2&\frac{3}{2}b_1+2b_2-e_1-\frac{1}{2}e_2
    \end{pmatrix} = \begin{pmatrix}
        0&0\\
        0&0
    \end{pmatrix},\\
    M_{Q_5, 3}& = \begin{pmatrix}
        2c_1+c_2+f&c_1+3c_2+\frac{f}{2}\\
        c_1+3c_2+\frac{f}{2}&3c_1+4c_2-f
    \end{pmatrix} = \begin{pmatrix}
        2&1\\
        1&3
    \end{pmatrix},
\end{align*}
thus
\[
    M_{Q_5} = \begin{pmatrix}
        1&-2&0&0\\
        -2&9&0&0\\
        0&0&2&1\\
        0&0&1&3
    \end{pmatrix} \sim \begin{pmatrix}
        1&0&0&0\\
        0&5&0&0\\
        0&0&2&1\\
        0&0&1&3
    \end{pmatrix}.
\]
This quadratic form is positive definite and universal by the 15-Theorem. Alternatively, it is possible to consult \cite[Table~5]{Bh}, which contains the equivalent form
\[
    \begin{pmatrix}
        1&0&0&0\\
        0&2&1&0\\
        0&1&3&0\\
        0&0&0&5
    \end{pmatrix}
\]
encoded as ``$ 25:\ 2\ 3\ 5\ 0\ 0\ 2 $".
\end{proof}

\begin{lemma}
Let $ K = \Q(\sqrt{6}) $ and $ G_6 $ be the generalized form with the matrix
\[
    M_{G_6} = \begin{pmatrix}
        3&\sqrt{6}&-\frac{5}{2}&\frac{\sqrt{6}}{2}\\
        \sqrt{6}&3&-\frac{\sqrt{6}}{2}&-\frac{1}{2}\\
        -\frac{5}{2}&-\frac{\sqrt{6}}{2}&3&-\sqrt{6}\\
        \frac{\sqrt{6}}{2}&-\frac{1}{2}&-\sqrt{6}&3
    \end{pmatrix}.
\]
The generalized form $ G_6 $ is $ \Z $-universal over $ K $.
\end{lemma}
\begin{proof}
The coefficients of $ G_6 $ are $ a = 3 $, $ b = 2\sqrt{6} $, $ c = 3 $, $ d = -5 $, $ e = \sqrt{6} $, and $ f = -1 $. Let $ Q_6 $ be the quaternary quadratic form $ M_{Q_6}(x_1, y_1, x_2, y_2) = G_6(x_1+y_1\sqrt{6}, x_2+y_2\sqrt{6}) $. By \Cref{lemmaMQBin} with $ D = 6 $,
\[
    M_{Q_6} = \begin{pmatrix}
        2a_1+d&12a_2&b_1+e_1&6(b_2-e_2)\\
        12a_2&6(2a_1-d)&6(b_2+e_2)&6(b_1-e_1)\\
        b_1+e_1&6(b_2+e_2)&2c_1+f&12c_2\\
        6(b_2-e_2)&6(b_1-e_1)&12c_2&6(2c_1-f)
    \end{pmatrix} = \begin{pmatrix}
        1&0&0&6\\
        0&66&18&0\\
        0&18&5&0\\
        6&0&0&42
    \end{pmatrix}.
\]
We get
\[
    M_{Q_6} \sim \begin{pmatrix}
        1&0&0&0\\
        0&66&18&0\\
        0&18&5&0\\
        0&0&0&6
    \end{pmatrix} \sim \begin{pmatrix}
        1&0&0&0\\
        0&3&3&0\\
        0&3&5&0\\
        0&0&0&6
    \end{pmatrix} \sim \begin{pmatrix}
        1&0&0&0\\
        0&3&0&0\\
        0&0&2&0\\
        0&0&0&6
    \end{pmatrix}.
\]
The last quadratic form is positive definite and universal by the 15-Theorem.
\end{proof}

\begin{lemma}
Let $ K = \Q(\sqrt{7}) $ and $ G_7 $ be the generalized quadratic form with the matrix
\[
    M_{G_7} = \begin{pmatrix}
        3&\frac{1+2\sqrt{7}}{2}&-\frac{5}{2}&\frac{-1+2\sqrt{7}}{2}\\
        \frac{1+2\sqrt{7}}{2}&4&\frac{-1-2\sqrt{7}}{2}&\frac{3}{2}\\
        -\frac{5}{2}&\frac{-1-2\sqrt{7}}{2}&3&\frac{1-2\sqrt{7}}{2}\\
        \frac{-1+2\sqrt{7}}{2}&\frac{3}{2}&\frac{1-2\sqrt{7}}{2}&4
    \end{pmatrix}.
\]
The generalized form $ G_7 $ is $ \Z $-universal over $ K $.
\end{lemma}
\begin{proof}
The coefficients of $ G_7 $ are $ a = 3 $, $ b = 1+2\sqrt{7} $, $ c = 4 $, $ d = -5 $, $ e = -1+2\sqrt{7} $, and $ f = 3 $. By \Cref{lemmaMQBin} with $ D = 7 $,
\[
    M_{Q_7} = \begin{pmatrix}
        2a_1+d&14a_2&b_1+e_1&7(b_2-e_2)\\
        14a_2&7(2a_1-d)&7(b_2+e_2)&7(b_1-e_1)\\
        b_1+e_1&7(b_2+e_2)&2c_1+f&14c_2\\
        7(b_2-e_2)&7(b_1-e_1)&14c_2&7(2c_1-f)
    \end{pmatrix} = \begin{pmatrix}
        1&0&0&0\\
        0&77&28&14\\
        0&28&11&0\\
        0&14&0&35
    \end{pmatrix}.
\]
Let
\[
    A_7 = \begin{pmatrix}
        1&0&0&0\\
        0&2&-3&3\\
        0&-5&8&-7\\
        0&-1&1&-1
    \end{pmatrix}.
\]
It can be verified that $ \det(A_7) = 1 $, hence $ A_7 \in \SL_4(\Z) $. The equivalent form with the matrix
\[
    A_7^\top M_{Q_7} A_7 = \begin{pmatrix}
        1&0&0&0\\
        0&2&-3&3\\
        0&-5&8&-7\\
        0&-1&1&-1
    \end{pmatrix}^\top \begin{pmatrix}
        1&0&0&0\\
        0&77&28&14\\
        0&28&11&0\\
        0&14&0&35
    \end{pmatrix}\begin{pmatrix}
        1&0&0&0\\
        0&2&-3&3\\
        0&-5&8&-7\\
        0&-1&1&-1
    \end{pmatrix} = \begin{pmatrix}
        1&0&0&0\\
        0&2&1&0\\
        0&1&4&0\\
        0&0&0&7
    \end{pmatrix}
\]
is positive definite and universal by the 15-Theorem. Alternatively, it can be found in \cite[Table 5]{Bh} as ``49:\ 2\ 4\ 7\ 0\ 0\ 2".
\end{proof}

\begin{lemma}
Let $ K = \Q(\sqrt{10}) $. If $ G_{10} $ is the generalized quadratic form with the matrix
\[
    M_{G_{10}} = \begin{pmatrix}
        3&\frac{\sqrt{10}}{2}&-\frac{5}{2}&\frac{\sqrt{10}}{2}\\
        \frac{\sqrt{10}}{2}&5&-\frac{\sqrt{10}}{2}&\frac{7}{2}\\
        -\frac{5}{2}&-\frac{\sqrt{10}}{2}&3&\frac{-\sqrt{10}}{2}\\
        \frac{\sqrt{10}}{2}&\frac{7}{2}&\frac{\sqrt{10}}{2}&5
    \end{pmatrix},
\]
then $ G_{10} $ is $ \Z $-universal over $ K $.
\end{lemma}
\begin{proof}
The coefficients of $ G_{10} $ are $ a = 3 $, $ b = \sqrt{10} $, $ c = 5 $, $ d = -5 $, $ e = \sqrt{10} $, and $ f = 7 $. Let $ Q_{10} $ be the quaternary quadratic form $ Q_{10}(x_1, y_1, x_2, y_2) = G_{10}(x_1+y_1\sqrt{10}, x_2+y_2\sqrt{10}) $. By \Cref{lemmaMQBin} with $ D = 10 $,
\[
    M_{Q_{10}} = \begin{pmatrix}
        2a_1+d&20a_2&b_1+e_1&10(b_2-e_2)\\
        20a_2&10(2a_1-d)&10(b_2+e_2)&10(b_1-e_1)\\
        b_1+e_1&10(b_2+e_2)&2c_1+f&20c_2\\
        10(b_2-e_2)&10(b_1-e_1)&20c_2&10(2c_1-f)
    \end{pmatrix} = \begin{pmatrix}
        1&0&0&0\\
        0&110&20&0\\
        0&20&17&20\\
        0&0&20&30
    \end{pmatrix}.
\]
Let
\[
    A_{10} = \begin{pmatrix}
        1&0&0&0\\
        0&1&1&2\\
        0&-6&-5&-10\\
        0&4&3&7
    \end{pmatrix}.
\]
It can be verified that $ \det(A_{10}) = 1 $, hence $ A \in \SL_4(\Z) $. The equivalent form with the matrix
\[
    A_{10}^\top M_{Q_{10}}A_{10} = \begin{pmatrix}
        1&0&0&0\\
        0&1&1&2\\
        0&-6&-5&-10\\
        0&4&3&7
    \end{pmatrix}^\top \begin{pmatrix}
        1&0&0&0\\
        0&110&20&0\\
        0&20&17&20\\
        0&0&20&30
    \end{pmatrix} \begin{pmatrix}
        1&0&0&0\\
        0&1&1&2\\
        0&-6&-5&-10\\
        0&4&3&7
    \end{pmatrix} = \begin{pmatrix}
        1&0&0&0\\
        0&2&0&0\\
        0&0&5&0\\
        0&0&0&10
    \end{pmatrix}
\]
is positive definite and universal by the 15-Theorem.
\end{proof}

\section{Indefinite generalized quadratic forms}
\label{secIndef2}

This section elaborates on the proof of the universality of $ G(z, w) = z\tau(z) +w\tau(w) $ given in \Cref{secIndef} as \Cref{thmIndef}.

One of the classical methods for studying representations by quadratic forms involves using the \emph{local-global principle.} We will use the following version of the local-global principle for indefinite forms. If $ p $ is a prime number, then $ \Z_p $ denotes the ring of $ p $-adic integers.

\begin{theorem}[{\cite[Chap. 9, Thm. 1.5]{Ca}}]
\label{thmLG}
Let $ Q $ be a regular indefinite integral form in $ n \geq 4 $ variables and let $ a \neq 0 $ be an integer. Suppose that $ a $ is represented by $ Q $ over all $ \Z_p $. Then $ a $ is represented by $ Q $ over $ \Z $.
\end{theorem}

Recall that by our definition, every quadratic form $ Q $ represents $ 0 $ (trivially). To represent nonzero $ a $, we initially find the representation modulo $ p $, followed by lifting it to a representation over $ \mathbb{Z}_p $ using the following corollaries of Hensel's lemma. The tuple $ (x_1, x_2, \dots, x_n) \in \Z_p^n $ is called \emph{primitive} if at least one of the $ x_i $ is invertible in $ \Z_p $.

\begin{lemma}[{\cite[Chap. 2, Cor. 2]{Se}}]
\label{lemmaLift1}
Suppose $ p \neq 2 $. Let $ Q(x) = \sum a_{ij}x_ix_j $ with $ a_{ij} = a_{ji} $ be a quadratic form with coefficients in $ \Z_p $ whose determinant $ \det(a_{ij}) $ is invertible. Let $ a \in \Z_p $. Every primitive solution of the equation $ Q(x) \equiv a \pmod{p} $ lifts to a true solution.
\end{lemma}

\begin{lemma}[{\cite[Chap. 2, Cor. 3]{Se}}]
\label{lemmaLift2}
Suppose $ p = 2 $. Let $ Q(x) = \sum a_{ij}x_ix_j $ with $ a_{ij} = a_{ji} $ be a quadratic form with coefficients in $ \Z_2 $ and let $ a \in \Z_2 $. Let $ x $ be a primitive solution of $ Q(x) \equiv a \pmod{8} $. We can lift $ x $ to a true solution provided $ x $ does not annihilate all the $ \frac{\partial Q}{\partial x_j} $ modulo $ 4 $; this last condition is fulfilled if $ \det(a_{ij}) $ is invertible.
\end{lemma}

We proceed with proving that $ x^2+y^2-Dz^2-Dw^2 $ represents every $ a \in \Z $.
\begin{lemma}
\label{lemmaPOdd}
Let $ p \neq 2 $ be a prime, $ D \in \Z $ squarefree, and $ a \in \Z_p $. The quadratic form $ x^2+y^2-Dz^2-Dw^2 $ represents $ a $ over $ \Z_p $.
\end{lemma}
\begin{proof}
We can assume $ a \neq 0 $. First, suppose $ a = p^{2k}a_1 $ where $ k \in \Z_{\geq 0} $ and $ p \nmid a_1 $. The congruence
\[
    x^2+y^2 \equiv a_1 \pmod{p}
\]
has a solution, hence by \Cref{lemmaLift1}, there exist $ x_1, y_1 \in \Z_p $ such that
\[
    x_1^2+y_1^2 = a_1.
\]
The tuple $ (x, y, z, w) = (p^kx_1, p^ky_1, 0, 0) $ is a solution of the equation $x^2+y^2-D z^2- D w^2 = a$.

Secondly, suppose $ a = p^{2k+1}a_1 $ where $ k \in \Z_{\geq 0} $ and $ p \nmid a_1 $. We distinguish two cases.

If $ p \nmid D $, then the congruence
\[
    x^2+y^2 \equiv pa_1+D \pmod{p}
\]
has a solution, hence by \Cref{lemmaLift1}, there exist $ x_1, y_1 \in \Z_p $ such that
\[
    x_1^2+y_1^2 = pa_1+D.
\]
The tuple $ (x, y, z, w) = (p^kx_1, p^ky_1, p^k, 0) $ is a solution of the original equation because
\[
    x^2+y^2-Dz^2-Dw^2 = p^{2k}(x_1^2+z_1^2-D) = p^{2k}\cdot pa_1 = a.
\]

If $ p \mid D $, let $ D_p = \frac{D}{p} $. We have $ p \nmid D_p $ because $ D $ is squarefree. The congruence
\[
    -D_p(z^2+w^2) \equiv a_1 \pmod{p}
\]
has a solution, hence by \Cref{lemmaLift1}, there exist $ z_1, w_1 \in \Z_p $ such that
\[
    -D_p(z_1^2+w_1^2) = a_1.
\]
The tuple $ (x, y, z, w) = (0, 0, p^kz_1, p^kw_1) $ is a solution because
\[
    -Dz^2-Dw^2 = p^{2k+1}(-D_p)(z_1^2+w_1^2) = p^{2k+1}a_1 = a.\qedhere
\]
\end{proof}

\begin{lemma}
\label{lemmaP2}
Let $ D \in \Z $ be squarefree and $ a \in \Z_2 $. The quadratic form $ x^2+y^2-Dz^2-Dw^2 $ represents $ a $ over~$ \Z_2 $.
\end{lemma}
\begin{proof}
We can assume $ a \neq 0 $. First, suppose that $ 4 \nmid a $, hence $ a \equiv 1, 2, 3, 5, 6, 7 \pmod{8} $. We find a solution to the congruence
\[
    x^2+y^2-Dz^2-Dw^2 \equiv a \pmod{8}
\]
where $ x $ is odd, hence $ \frac{\partial Q}{\partial x} = 2x \neq 0 \pmod{4} $. This solution lifts to a true solution in $ \Z_2 $ by \Cref{lemmaLift2}.

Let
\[
    (x, y) = \begin{cases}
        (1,0),&\text{if }a \equiv 1 \pmod{8},\\
        (1,1),&\text{if }a \equiv 2 \pmod{8},\\
        (1,2),&\text{if }a \equiv 5 \pmod{8}.
    \end{cases}
\]
We have $ x^2+y^2 \equiv a \pmod{8} $ in each case.

If $ a \equiv 3 \pmod{8} $, let
\[
    (z, w) = \begin{cases}
        (1,1),&\text{if }D \equiv 1,3,5 \pmod{8},\\
        (1,0),&\text{if }D \equiv 2,6,7 \pmod{8},
    \end{cases}
\]
if $ a \equiv 6 \pmod{8} $, let
\[
    (z, w) = \begin{cases}
        (2,0),&\text{if }D \equiv 1,3,5,7 \pmod{8},\\
        (1,1),&\text{if }D \equiv 2,6 \pmod{8},
    \end{cases}
\]
and if $ a \equiv 7 \pmod{8} $, let
\[
    (z, w) = \begin{cases}
        (1,1),&\text{if }D \equiv 1,3,5,7\pmod{8},\\
        (1,0),&\text{if }D \equiv 2,6 \pmod{8}.
    \end{cases}
\]
In each of the above cases, $ a+Dz^2+Dw^2 \pmod{8} \in \{1, 2, 5\} $, hence we can find $ (x, y) = (1, y) $ such that
\[
    x^2+y^2 \equiv a+Dz^2+Dw^2 \pmod{8}.
\]

Secondly, assume $ 4 \mid a $. Let us write $ a = 4^k a_1 $ where $ 4 \nmid a_1 $. By the first part of the proof, there exists a tuple $ (x_1, y_1, z_1, w_1) \in \Z_p^4 $ such that
\[
    x_1^2+y_1^2+Dz_1^2+Dw_1^2 = a_1
\]
and we set $ (x, y, z, w) = (2^kx_1, 2^ky_1, 2^kz_1, 2^kw_1) $.
\end{proof}

Now we can observe that in the proof of \Cref{thmIndef} the form $ Q $ represents $ a $ over $ \Z_p $ for $ p $ odd by \Cref{lemmaPOdd} and for $ p = 2 $ by \Cref{lemmaP2}. And by \Cref{thmLG}, it represents $ a $ over $ \Z $.

\begin{proposition}
\label{propRepZero}
Let $ D \in \Z_{\geq 1} $. The quadratic form
\[
    Q(x, y, z, w) = x^2+y^2-Dz^2-Dw^2
\]
represents $ 0 $ non-trivially if and only if $ D $ is a sum of two squares.
\end{proposition}
\begin{proof}
Assume that $ D $ is a sum of two squares, say $ D = a^2+b^2 $ where $ a, b \in \Z $. If $ (x, y, z, w) = (a, b, 1, 0) $, then $ Q(x, y, z, w) = 0 $.

Now we prove the converse under the assumption that $ D $ is squarefree. Let $ p \mid D $ be an odd prime and $ (x, y, z, w) $ be a primitive solution to
\[
    x^2+y^2-Dz^2-Dw^2 = 0.
\]
In particular, not all of $ x $, $ y $, $ z $, $ w $ are divisible by $ p $.

First, suppose that $ p \nmid x $ or $ p \nmid y $. Without loss of generality, we can assume $ p \nmid y $. Reducing the equation modulo $ p $, we get $ x^2+y^2 \equiv 0 \pmod{p} $, hence
\[
    \left(\frac{x}{y}\right)^2 \equiv -1 \pmod{p}.
\]
This shows that $ - 1 $ is a square modulo $ p $, hence $ p \equiv 1 \pmod{4} $.

Secondly, suppose that $ p \mid x $ and $ p \mid y $. From the equation, we have $ p^2 \mid D(z^2+w^2) $, and because $ D $ is squarefree, $ p \mid z^2+w^2 $. Since $ (x, y, z, w) $ is primitive, we get $ p \nmid z $ or $ p \nmid w $, and the congruence $ z^2+w^2 \equiv 0 \pmod{p} $ implies $ p \equiv 1 \pmod{4} $. We proved that every odd prime $ p \mid D $ satisfies $ p \equiv 1 \pmod{4} $, hence $ D $ is a sum of two squares.

To prove the converse for general $ D $, write $ D = k^2d $ where $ k, d \in \Z $ and $ d $ is squarefree. If $ (x_0, y_0, z_0, w_0) $ is a solution to
\[
    x^2+y^2-k^2dz^2-k^2dw^2=0,
\]
then $ (x_0, y_0, kz_0, kw_0) $ is a solution to
\[
    x^2+y^2-dz^2-dw^2=0.
\]
Thus, if $ Q $ represents $ 0 $, then $ x^2+y^2-dz^2-dw^2 $ also represents $ 0 $. By the first part of the proof, $ d $ is a sum of two squares, hence $ D $ is also a sum of two squares.
\end{proof}

\section{Examples of universal binary generalized forms}
\label{secExamples2}

\subsection{Preliminaries}
\label{subsecPreliminaries}

Let $ K = \Q(\sqrt{D}) $ where $ D \in \Z_{\geq 2} $ is squarefree, $ G $ be a $ \Z $-valued generalized form over $ K $ in $ n $ variables, and $ Q $ be the quadratic form in $ 2n $ variables associated to $ G $. By \Cref{lemmaGQ}, if one of the following conditions is satisfied:
\begin{enumerate}[i)]
    \item $ D \equiv 2, 3 \pmod{4} $ and $ G $ is integral,
    \item $ D \equiv 1 \pmod{4} $ and $ G $ is classical,
\end{enumerate}
then $ Q $ is classical. By \Cref{thmBinary}, a $ \Z $-universal binary generalized form satisfying the conditions above exists over $ K $ if and only if $ D \in \{2,3,5,6,7,10\} $.

A complete list of universal classical positive definite quaternary quadratic forms (up to equivalence) is available \cite[Table~5]{Bh}. In this section, we answer the following question: Which of these quadratic forms are equivalent to a quadratic form $ Q $ associated to $ G $ satisfying the conditions above? For each $ D \in \{2,3,5,6,7,10\} $ we identify all such forms.

We use Bhargava's notation $D:\ a\ b\ c\ d\ e\ f$ for a ternary quadratic form with matrix
\[
    M_L = \begin{pmatrix}
        a&\frac{f}{2}&\frac{e}{2}\\
        \frac{f}{2}&b&\frac{d}{2}\\
        \frac{e}{2}&\frac{d}{2}&c
    \end{pmatrix}
\]
of determinant $ D $. If $ L = D:\ a\ b\ c\ d\ e\ f$, then $ 1 \oplus L $ denotes the quaternary form with matrix
\[
    M_{1\oplus L} = \begin{pmatrix}
        1&0&0&0\\
        0&a&\frac{f}{2}&\frac{e}{2}\\
        0&\frac{f}{2}&b&\frac{d}{2}\\
        0&\frac{e}{2}&\frac{d}{2}&c
    \end{pmatrix}.
\]
We also need a concise way of representing a generalized binary form. Thus we let $ [a,b,c,d,e,f] $ represent the binary generalized form with matrix
\[
    \begin{pmatrix}
        a&\frac{b}{2}&\frac{d}{2}&\frac{e}{2}\\
        \frac{b}{2}&c&\frac{\tau(e)}{2}&\frac{f}{2}\\
        \frac{d}{2}&\frac{\tau(e)}{2}&\tau(a)&\frac{\tau(b)}{2}\\
        \frac{e}{2}&\frac{f}{2}&\frac{\tau(b)}{2}&\tau(c)
    \end{pmatrix}
\]
where $ a, b, c, e \in \O_K $ and $ d, f \in \Z $. We also use $ [a_{11}, a_{12}, a_{13}, a_{14}, a_{22}, a_{23}, a_{24}, a_{33}, a_{34}, a_{44}] $ to represent the quaternary quadratic form with matrix
\[
    \begin{pmatrix}
        a_{11}&\frac{a_{12}}{2}&\frac{a_{13}}{2}&\frac{a_{14}}{2}\\
        \frac{a_{12}}{2}&a_{22}&\frac{a_{23}}{2}&\frac{a_{24}}{2}\\
        \frac{a_{13}}{2}&\frac{a_{23}}{2}&a_{33}&\frac{a_{34}}{2}\\
        \frac{a_{14}}{2}&\frac{a_{24}}{2}&\frac{a_{34}}{2}&a_{44}
    \end{pmatrix}.
\]

\subsection{Universal quadratic forms associated to binary generalized forms for \texorpdfstring{$ D=2 $}{D=2}}
\label{subsecExamplesD2}

Let $ G $ be a $ \Z $-universal binary generalized form over $ \Q(\sqrt{2}) $, $ Q $ be the associated quaternary quadratic form, and $ M_Q $ be the matrix of $ Q $. We have $ 4 \mid \det(M_Q) $ by \Cref{lemmaDetMQ} and $ \frac{\det(M_Q)}{4} \equiv 0, 1 \pmod{4} $ by \Cref{propDetBinary}. From the table of universal classical quaternary forms \cite[Table 5]{Bh}, we select all entries satisfying these two conditions. We also know that $ \rank(M_Q \bmod 2) \leq 2 $ by \Cref{propRank}, hence we compute the rank of each of them. The results are in \Cref{tabD2a}. We see that we are left with $ 23 $ possible candidates.

\begin{table}[ht]
    \centering
    \begin{tabular}{|l|r|}
        \hline
        $L$&$\rank(M_{1\oplus L} \bmod{2})$\\
        \hline
        $4:\ 1\ 1\ 4\ 0\ 0\ 0$&$3$\\
        $4:\ 1\ 2\ 2\ 0\ 0\ 0$&$2$\\
        $4:\ 2 \ 2\ 2\ 2\ 2\ 0$&$3$\\
        $16:\ 1\ 2\ 8\ 0\ 0\ 0$&$2$\\
        $16:\ 2\ 2\ 4\ 0\ 0\ 0$&$1$\\
        $16:\ 2\ 3\ 3\ 2\ 0\ 0$&$2$\\
        $20:\ 1\ 2\ 10\ 0\ 0\ 0$&$2$\\
        $20:\ 2\ 2\ 5\ 0\ 0\ 0$&$2$\\
        $20:\ 2\ 2\ 6\ 2\ 2\ 0$&$3$\\
        $20:\ 2\ 3\ 4\ 0\ 0\ 2$&$3$\\
        $20:\ 2\ 4\ 4\ 4\ 2\ 0$&$3$\\
        $32:\ 2\ 4\ 4\ 0\ 0\ 0$&$1$\\
        $32:\ 2\ 4\ 5\ 4\ 0\ 0$&$2$\\
        $36:\ 2\ 3\ 6\ 0\ 0\ 0$&$2$\\
        $36:\ 2\ 4\ 5\ 0\ 2\ 0$&$3$\\
        $36:\ 2\ 4\ 6\ 4\ 2\ 0$&$3$\\
        $36:\ 2\ 5\ 5\ 4\ 2\ 2$&$3$\\
        $48:\ 2\ 3\ 8\ 0\ 0\ 0$&$2$\\
        $48:\ 2\ 4\ 6\ 0\ 0\ 0$&$1$\\
        $48:\ 2\ 5\ 5\ 2\ 0\ 0$&$2$\\
        $52:\ 2\ 3\ 9\ 2\ 0\ 0$&$2$\\
        $52:\ 2\ 5\ 6\ 2\ 0\ 2$&$3$\\
        $52:\ 2\ 5\ 6\ 4\ 0\ 0$&$2$\\
        $64:\ 2\ 4\ 8\ 0\ 0\ 0$&$1$\\
        $68:\ 2\ 4\ 9\ 0\ 2\ 0$&$3$\\
        $68:\ 2\ 4\ 10\ 4\ 2\ 0$&$3$\\
        $68:\ 2\ 5\ 7\ 2\ 0\ 0$&$2$\\
        $80:\ 2\ 4\ 10\ 0\ 0\ 0$&$1$\\
        $80:\ 2\ 4\ 11\ 4\ 0\ 0$&$2$\\
        $80:\ 2\ 5\ 8\ 0\ 0\ 0$&$2$\\
        $96:\ 2\ 4\ 12\ 0\ 0\ 0$&$1$\\
        $96:\ 2\ 4\ 13\ 4\ 0\ 0$&$2$\\
        $100:\ 2\ 4\ 13\ 0\ 2\ 0$&$3$\\
        $100:\ 2\ 4\ 14\ 4\ 2\ 0$&$3$\\
        $100:\ 2\ 5\ 10\ 0\ 0\ 0$&$2$\\
        $112:\ 2\ 4\ 14\ 0\ 0\ 0$&$1$\\
        \hline
    \end{tabular}
    \caption{Ternary forms $ L $ such that $ 1\oplus L $ is universal, $ \det(M_{1\oplus L}) $ is divisible by $ 4 $, and $ \frac{\det(M_{1\oplus L})}{4} \equiv 0, 1 \pmod{4} $.}
    \label{tabD2a}
\end{table}

To determine which of these candidates are equivalent to a quadratic form associated to some $ G $, we use the \emph{$2$-adic normal form}. Let $ \NF(M) $ denote the normal form of a symmetric matrix $ M $ over $ \Z_2 $. The precise definition is very technical and we do not give it here. It is sufficient for us to be able to compute $ \NF(M) $, for which we use SageMath. Details on the implementation of the $ p $-adic normal form can be found in the reference manual \cite{Sa}. The important property of the normal form is uniqueness: Let $ Q_1 $ and $ Q_2 $ be two quadratic forms in $ n $ variables over $ \Z_2 $ with matrices $ M_{Q_1} $ and $ M_{Q_2} $, respectively. Then $ Q_1 $ and $ Q_2 $ are equivalent over $ \Z_2 $ if and only if $ \NF(M_{Q_1}) = \NF(M_{Q_2}) $.

\begin{lemma}
\label{lemmaNF}
Let $ K = \Q(\sqrt{2}) $ and let $ G $ be an integral binary generalized form over $ K $ such that $ G(\mathbf{a}) \in \Q $ for every $ \mathbf{a} \in K^2 $. Let $ Q $ be the quaternary quadratic form associated to $ G $ and $ M_Q $ be the matrix of $ Q $. We have $ \NF(M_Q) \bmod{8} = \NF(M) \bmod{8} $ for some matrix $ M $ such that
\begin{equation}
\label{eqM}
    M = \begin{pmatrix}
        A_1&4a_2&B_1&2E_1\\
        4a_2&2A_2&2E_2&2B_2\\
        B_1&2E_2&C_1&4c_2\\
        2E_1&2B_2&4c_2&2C_2
    \end{pmatrix}
\end{equation}
where $ A_1, B_1, C_1 \in \{0, \dots, 7\} $, $ a_2, c_2 \in \{0, 1\} $, $ A_2, C_2, B_2, E_1, E_2 \in \{0, \dots, 3\} $, $ A_1+A_2 \equiv 0 \pmod{4} $, $ C_1+C_2 \equiv 0 \pmod{4} $, $ B_1+B_2 \equiv 0 \pmod{2} $, and $ E_1+E_2 \equiv 0 \pmod{2} $.
\end{lemma}
\begin{proof}
By \Cref{lemmaMQBin} with $ D = 2 $, the matrix of $ Q $ equals
\[
    M_{Q} = \begin{pmatrix}
        2a_1+d&4a_2&b_1+e_1&2(b_2-e_2)\\
        4a_2&2(2a_1-d)&2(b_2+e_2)&2(b_1-e_1)\\
        b_1+e_1&2(b_2+e_2)&2c_1+f&4c_2\\
        2(b_2-e_2)&2(b_1-e_1)&4c_2&2(2c_1-f)
     \end{pmatrix}.
\]
Let $ A_1 = 2a_1+d $, $ A_2 = 2a_1-d $, $ C_1 = 2c_1+f $, $ C_2 = 2c_1-f $, $ B_1 = b_1+e_1 $, $ B_2 = b_1-e_1 $, $ E_1 = b_2-e_2 $, and $ E_2 = b_2+e_2 $. These parameters clearly satisfy the stated congruences. Since $ \NF(M_Q) \bmod{8} $ depends only on $ M_Q \bmod{8} $, we can reduce $ M_Q $ modulo $ 8 $ to obtain one of the matrices \eqref{eqM}.
\end{proof}

We compute $ \NF(M_{1\oplus L}) \bmod{8} $ for each of the forms in \Cref{tabD2a} satisfying $ \rank(1\oplus L \bmod{2}) \leq 2 $ and check whether it equals $ \NF(M) \bmod{8} $ for one of the finitely many matrices $ M $ in \Cref{lemmaNF}. If not, then we know that $ 1\oplus L $ is not equivalent to a quadratic form associated to some $ G $ and we can discard it. The SageMath code of this computation is available at: \url{https://github.com/zinmik/Generalized-forms}.

We are left with $ 15 $ quadratic forms $ 1\oplus L $ and for each one, we directly find $ G $ such that the quadratic form associated to $ G $ is equivalent to $ 1\oplus L $, see \Cref{tabD2b}.

\begin{table}[ht]
    \centering
    \begin{tabular}{|l|l|l|}
        \hline
        $L$&$G$&$Q$\\
        \hline
        $4:\ 1\ 2\ 2\ 0\ 0\ 0$&$[1,-\sqrt{2},1,-1,0,-1]$&$[1,0,0,-4,6,-4,0,1,0,6]$\\
        $16:\ 2\ 2\ 4\ 0\ 0\ 0$&$[1,-1,1,-1,1,0]$&$[1,0,0,0,6,0,-8,2,0,4]$\\
        $20:\ 1\ 2\ 10\ 0\ 0\ 0$&$[1,-\sqrt{2},2,-1,-\sqrt{2},-1]$&$[1,0,0,0,6,-8,0,3,0,10]$\\
        $20:\ 2\ 2\ 5\ 0\ 0\ 0$&$[1,-\sqrt{2},2,1,\sqrt{2},1]$&$[3,0,0,-8,2,0,0,5,0,6]$\\
        $32:\ 2\ 4\ 4\ 0\ 0\ 0$&$[1,-1-\sqrt{2},2,-1,1,-2]$&$[1,0,0,-4,6,-4,-8,2,0,12]$\\
        $36:\ 2\ 3\ 6\ 0\ 0\ 0$&$[1,0,1,-1,0,1]$&$[1,0,0,0,6,0,0,3,0,2]$\\
        $48:\ 2\ 4\ 6\ 0\ 0\ 0$&$[1,0,1,-1,0,0]$&$[1,0,0,0,6,0,0,2,0,4]$\\
        $52:\ 2\ 3\ 9\ 2\ 0\ 0$&$[1,-\sqrt{2},2,1,0,-1]$&$[3,0,0,-4,2,-4,0,3,0,10]$\\
        $52:\ 2\ 5\ 6\ 4\ 0\ 0$&$[1,-\sqrt{2},2,-1,0,1]$&$[1,0,0,-4,6,-4,0,5,0,6]$\\
        $64:\ 2\ 4\ 8\ 0\ 0\ 0$&$[1,-\sqrt{2},2,-1,0,-2]$&$[1,0,0,-4,6,-4,0,2,0,12]$\\
        $68:\ 2\ 5\ 7\ 2\ 0\ 0$&$[2,-2\sqrt{2},2,-1,0,1]$&$[3,0,0,-8,10,-8,0,5,0,6]$\\
        $80:\ 2\ 4\ 10\ 0\ 0\ 0$&$[1,-1,2,0,-1,-1]$&$[2,0,-4,0,4,0,0,3,0,10]$\\
        $96:\ 2\ 4\ 12\ 0\ 0\ 0$&$[1,-1+\sqrt{2},3-\sqrt{2},-1,1+\sqrt{2},-2]$&$[1,0,0,0,6,8,-8,4,-8,16]$\\
        $100:\ 2\ 5\ 10\ 0\ 0\ 0$&$[2,-\sqrt{2},2,-1,\sqrt{2},1]$&$[3,0,0,-8,10,0,0,5,0,6]$\\
        $112:\ 2\ 4\ 14\ 0\ 0\ 0$&$[2,0,1,-3,0,0]$&$[1,0,0,0,14,0,0,2,0,4]$\\
        \hline
    \end{tabular}
    \caption{Quadratic forms associated to $\Z$-universal binary generalized forms over $ \Q(\sqrt{2}) $.}
    \label{tabD2b}
\end{table}

We proved the following proposition.
\begin{proposition}
\label{propExamplesD2}
Let $ K = \Q(\sqrt{2}) $ and let $ G $ be a $ \Z $-universal binary generalized form over $ K $. If $ Q $ is the quadratic form associated to $ G $, then $ Q $ is equivalent to $ 1 \oplus L $ where $ L $ is one of the $ 15 $ ternary forms in \Cref{tabD2b}. Conversely, for each of these forms, there exists an equivalent quadratic form associated to a $ \Z $-universal binary form over $ K $.
\end{proposition}

\subsection{Universal quadratic forms associated to binary generalized forms for \texorpdfstring{$ D=3 $}{D=3}}
\label{subsecExamplesD3}

\begin{lemma}
\label{lemmaRank3mod4}
Let $ K = \Q(\sqrt{D}) $ where $ D \in \Z\setminus\{0,1\} $ is squarefree. Let $ G $ be an integral binary generalized form over $ K $ and $ Q $ be the quadratic form associated to $ G $ with matrix $ M_Q $. Assume that $ G(\mathbf{a}) \in \Q $ for every $ \mathbf{a} \in K^2 $. If $ D \equiv 3 \pmod{4} $, then $ \rank(M_Q \bmod{2}) \in \{0, 2, 4\} $.
\end{lemma}
\begin{proof}
Let the matrix of $ G $ be given by \eqref{eqMGBinary}. By \Cref{lemmaMQBin},
\[
    M_Q = \begin{pmatrix}
        2a_1+d&2Da_2&b_1+e_1&D(b_2-e_2)\\
        2Da_2&D(2a_1-d)&D(b_2+e_2)&D(b_1-e_1)\\
        b_1+e_1&D(b_2+e_2)&2c_1+f&2Dc_2\\
        D(b_2-e_2)&D(b_1-e_1)&2Dc_2&D(2c_1-f)
    \end{pmatrix}.
\]
Let $ B = b_1+e_1 $ and $ E = b_2+e_2 $. Since $ D \equiv 1 \pmod{2} $, we get
\[
    M_Q \bmod{2} = \begin{pmatrix}
        d&0&B&E\\
        0&d&E&B\\
        B&E&f&0\\
        E&B&0&f
    \end{pmatrix}.
\]
There are $16$ choices for $ (d, B, E, f) $ modulo $ 2 $ and it can be checked that $ \rank(M_Q \bmod{2}) \in \{0,2,4\} $ for each of them.
\end{proof}

Let $ Q $ be a classical quaternary positive definite form over $ \Z $ with matrix $ M_Q $ and suppose that there exists a binary integral generalized form $ G $ over $ K = \Q(\sqrt{3}) $ such that $ Q $ is associated to $ G $. We know from \Cref{corDetMQ} that $ 3^2 \mid \det(M_Q) $. Moreover, by \Cref{propDetBinary},
\[
    \frac{\det(M_Q)}{3^2} \equiv 0, 1 \pmod{4}.
\]
By \Cref{propRank}, $ \rank(M_Q \bmod{3}) \leq 2 $ and by \Cref{lemmaRank3mod4}, $ \rank(M_Q \bmod{2}) $ is even.

In \Cref{tabD3}, we list all universal classical quaternary quadratic forms $ 1\oplus L $ with $ \det(M_{1\oplus L}) $ divisible by $ 9 $ and such that $ \frac{\det(M_{1\oplus L})}{9} \equiv 0, 1 \pmod{4} $ from \cite[Table 5]{Bh}. For each of them we compute its rank modulo $2$ and $3$.

\begin{table}[ht]
    \centering
    \begin{tabular}{|l|r|r|}
        \hline
         $L$&$\rank(M_{1\oplus L} \bmod{2})$&$\rank(M_{1\oplus L}\bmod{3})$\\
         \hline
         $9:\ 1\ 2\ 5\ 2\ 0\ 0$&$4$&$3$\\
         $9:\ 1\ 3\ 3\ 0\ 0\ 0$&$4$&$2$\\
         $9:\ 2\ 2\ 3\ 0\ 0\ 2$&$4$&$2$\\
         $36:\ 2\ 3\ 6\ 0\ 0\ 0$&$2$&$2$\\
         $36:\ 2\ 4\ 5\ 0\ 2\ 0$&$3$&$3$\\
         $36:\ 2\ 4\ 6\ 4\ 2\ 0$&$3$&$3$\\
         $36:\ 2\ 5\ 5\ 4\ 2\ 2$&$3$&$2$\\
         $45:\ 2\ 4\ 7\ 0\ 2\ 2$&$4$&$3$\\
         $45:\ 2\ 5\ 5\ 0\ 2\ 0$&$4$&$3$\\
         $45:\ 2\ 5\ 6\ 4\ 2\ 2$&$4$&$3$\\
         $72:\ 2\ 4\ 9\ 0\ 0\ 0$&$2$&$3$\\
         $72:\ 2\ 4\ 10\ 4\ 0\ 0$&$1$&$3$\\
         $72:\ 2\ 5\ 8\ 4\ 0\ 0$&$2$&$3$\\
         $90:\ 2\ 4\ 12\ 2\ 2\ 0$&$3$&$3$\\
         $90:\ 2\ 5\ 9\ 0\ 0\ 0$&$3$&$3$\\
         $108:\ 2\ 4\ 14\ 0\ 2\ 0$&$3$&$3$\\
         \hline
    \end{tabular}
    \caption{Ternary forms $ L $ such that $ 1\oplus L $ is universal, $ \det(M_{1\oplus L}) $ is divisible by $ 9 $ and $ \frac{\det(M_{1\oplus L})}{9} \equiv 0, 1 \pmod{4} $.}
    \label{tabD3}
\end{table}

\begin{proposition}
\label{propExamplesD3}
Let $ K = \Q(\sqrt{3}) $. If $ G $ is a $\Z$-universal binary generalized form over $ K $ and $ Q $ is the quadratic form associated to $ G $, then $ Q $ is equivalent to one of the following forms:
\[
    \begin{pmatrix}
        1&0&0&0\\
        0&1&0&0\\
        0&0&3&0\\
        0&0&0&3
    \end{pmatrix},\ 
    \begin{pmatrix}
        1&0&0&0\\
        0&2&1&0\\
        0&1&2&0\\
        0&0&0&3
    \end{pmatrix},\ 
    \begin{pmatrix}
        1&0&0&0\\
        0&2&0&0\\
        0&0&3&0\\
        0&0&0&6
    \end{pmatrix}.
\]
Conversely, for each of these forms, there exists an equivalent form $ Q $ associated to a $\Z$-universal binary generalized form over $ K $.
\end{proposition}
\begin{proof}
Let $ M_Q $ be the matrix of $ Q $. If $ Q $ is associated to a $\Z$-universal binary generalized form, then $ Q $ must be equivalent to one of the forms in \Cref{tabD3} by the discussion before the proposition. Moreover, $ \rank(M_Q \bmod{3}) \leq 2 $ and $ \rank(M_Q \bmod{2}) $ is even. The only forms $ L $ in \Cref{tabD3} such that $ \rank(M_{1\oplus L} \bmod{2}) $ is even and $ \rank(M_{1\oplus L} \bmod{3}) \leq 2 $ are $ 9:\ 1\ 3\ 3\ 0\ 0\ 0 $, $ 9:\ 2\ 2\ 3\ 0\ 0\ 2 $, and $ 36:\ 2\ 3\ 6\ 0\ 0\ 0 $.

Now we prove the converse. Let $ G_1 $ be the generalized form $ [2,-2\sqrt{3},2,-3,-\sqrt{3},0] $ and let $ Q_1 $ be the quadratic form associated to $ G_1 $. By \Cref{lemmaMQBin}, the matrix of $ Q_1 $ equals
\[
    M_{Q_1} = \begin{pmatrix}
        1&0&0&-3\\
        0&21&-9&0\\
        0&-9&4&0\\
        -3&0&0&12
    \end{pmatrix}.
\]
We verified in \Cref{lemmaExD3} that $ Q_1 $ is equivalent to $ 9:\ 1\ 3\ 3\ 0\ 0\ 0 $.

Let $ G_2 $ be the generalized form $ [1,-1-\sqrt{3},2,-1,1,-2] $ and let $ Q_2 $ be the quadratic form associated to $ G_2 $. The matrix of $ Q_2 $ equals
\[
    M_{Q_2} = \begin{pmatrix}
        1&0&0&-3\\
        0&9&-3&-6\\
        0&-3&2&0\\
        -3&-6&0&18
    \end{pmatrix}.
\]
If we let
\[
    A = \begin{pmatrix}
        1&-3&0&3\\
        0&-1&0&2\\
        0&-1&1&3\\
        0&-1&0&1
    \end{pmatrix},
\]
then $ \det(A)=1 $, hence $ A \in \SL_4(\Z) $. We have
\[
    A^\top M_{Q_2} A = \begin{pmatrix}
        1&-3&0&3\\
        0&-1&0&2\\
        0&-1&1&3\\
        0&-1&0&1
    \end{pmatrix}^\top\begin{pmatrix}
        1&0&0&-3\\
        0&9&-3&-6\\
        0&-3&2&0\\
        -3&-6&0&18
    \end{pmatrix}\begin{pmatrix}
        1&-3&0&3\\
        0&-1&0&2\\
        0&-1&1&3\\
        0&-1&0&1
    \end{pmatrix} = \begin{pmatrix}
        1&0&0&0\\
        0&2&1&0\\
        0&1&2&0\\
        0&0&0&3
    \end{pmatrix},
\]
hence $ Q_2 $ is equivalent to $ 9:\ 2\ 2\ 3\ 0\ 0\ 2 $.

Let $ G_3 $ be the generalized form $ [1,1,1,0,1,1] $ and let $ Q_3 $ be the quadratic form associated to $ G_3 $. The matrix of $ Q_3 $ equals
\[
    M_{Q_3} = \begin{pmatrix}
        2&0&2&0\\
        0&6&0&0\\
        2&0&3&0\\
        0&0&0&3
    \end{pmatrix} \sim \begin{pmatrix}
        1&0&0&0\\
        0&2&0&0\\
        0&0&3&0\\
        0&0&0&6
    \end{pmatrix},
\]
hence $ Q_3 $ is equivalent to $ 36:\ 2\ 3\ 6\ 0\ 0\ 0 $.
\end{proof}

\subsection{Universal quadratic forms associated to classical binary generalized forms for \texorpdfstring{$ D = 5 $}{D=5}}
\label{subsecExamplesD5}

Let $ Q $ be a classical quaternary positive definite form over $ \Z $ with matrix $ M_Q $ and suppose that there exists a binary classical generalized form $ G $ over $ K = \Q(\sqrt{5}) $ such that $ Q $ is associated to $ G $. By \Cref{corDetMQ}, $ 5^2 \mid \det(M_Q) $ and by \Cref{propRank}, $ \rank(M_Q \bmod{5}) \leq 2 $.

In \Cref{tabD5}, we list all universal quaternary quadratic forms $ 1\oplus L $ with $ \det(M_{1\oplus L}) $ divisible by $ 25 $ from \cite[Table 5]{Bh}. For each of the we compute its rank modulo $ 5 $.

\begin{table}[ht]
    \centering
    \begin{tabular}{|l|r|}
        \hline
         $L$&$\rank(M_{1\oplus L} \bmod{5})$\\
         \hline
         $25:\ 1\ 2\ 13\ 2\ 0\ 0$&$3$\\
         $25:\ 2\ 3\ 5\ 0\ 0\ 2$&$2$\\
         $25:\ 2\ 3\ 5\ 2\ 2\ 0$&$3$\\
         $50:\ 2\ 4\ 7\ 2\ 2\ 0$&$3$\\
         $100:\ 2\ 4\ 13\ 0\ 2\ 0$&$3$\\
         $100:\ 2\ 4\ 14\ 4\ 2\ 0$&$3$\\
         $100:\ 2\ 5\ 10\ 0\ 0\ 0$&$2$\\
         \hline
    \end{tabular}
    \caption{Ternary forms $ L $ such that $ 1\oplus L $ is universal and $ \det(M_{1\oplus L}) $ is divisible by $ 25 $.}
    \label{tabD5}
\end{table}

\begin{proposition}
\label{propExamplesD5}
Let $ K = \Q(\sqrt{5}) $. If $ G $ is a classical $ \Z $-universal binary generalized form over $ K $ and $ Q $ is a quadratic form associated to $ G $, then $ Q $ is equivalent to one of the following forms:
\[
    \begin{pmatrix}
        1&0&0&0\\
        0&2&1&0\\
        0&1&3&0\\
        0&0&0&5
    \end{pmatrix},\ 
    \begin{pmatrix}
        1&0&0&0\\
        0&2&0&0\\
        0&0&5&0\\
        0&0&0&10
    \end{pmatrix}.
\]
Conversely, for each of these forms, there exists an equivalent form $ Q $ associated to a classical $ \Z $-universal binary generalized form over $ K $.
\end{proposition}
\begin{proof}
Let $ M_Q $ be the matrix of $ Q $. If $ Q $ is associated to a classical $ \Z $-universal binary form, then $ Q $ must be one of the forms listed in \Cref{tabD5}. Moreover, $ \rank(M_Q \bmod{5}) \leq 2 $. The only forms $ L $ in \Cref{tabD5} such that $ \rank(M_{1\oplus L} \bmod{5}) \leq 2 $ are $ 25:\ 2\ 3\ 5\ 0\ 0\ 2 $ and $ 100:\ 2\ 5\ 10\ 0\ 0\ 0$.

Let $ G_1 = [3-\omega_5,0,1,-4,0,0] $ and $ G_2 = [2+\omega_5,0,6-2\omega_5,-4,0,-8] $ where $ \omega_5 = \frac{1+\sqrt{5}}{2} $ and let $ Q_1 $, $ Q_2 $ be the quadratic forms associated to $ G_1 $ and $ G_2 $, respectively. By \Cref{lemmaMQBin} and \Cref{lemmaExD5}, the matrices of $ Q_1 $ and $ Q_2 $ are
\begin{align*}
    M_{Q_1}& = \begin{pmatrix}
        1&-2&0&0\\
        -2&9&0&0\\
        0&0&2&1\\
        0&0&1&3
    \end{pmatrix}\sim \begin{pmatrix}
        1&0&0&0\\
        0&2&1&0\\
        0&1&3&0\\
        0&0&0&5
    \end{pmatrix},\\
    M_{Q_2}& = \begin{pmatrix}
        1&3&0&0\\
        3&14&0&0\\
        0&0&2&-4\\
        0&0&-4&18
    \end{pmatrix}\sim\begin{pmatrix}
        1&0&0&0\\
        0&2&0&0\\
        0&0&5&0\\
        0&0&0&10
    \end{pmatrix}.
\end{align*}

\end{proof}

\subsection{Universal quadratic forms associated to binary generalized forms for \texorpdfstring{$ D = 6, 7, 10 $}{D=6,7,10}}
\label{subsecExamplesD6}

\begin{proposition}
\label{propExamplesD6}
Let $ K = \Q(\sqrt{6}) $. If $ G $ is a $ \Z $-universal binary generalized form over $ K $ and $ Q $ is the quadratic form associated to $ G $, then $ Q $ is equivalent to the form
\[
    \begin{pmatrix}
        1&0&0&0\\
        0&2&0&0\\
        0&0&3&0\\
        0&0&0&6
    \end{pmatrix}.
\]
Conversely, there exists an equivalent form $ Q $ associated to a $ \Z $-universal binary generalized form over $ K $.
\end{proposition}
\begin{proof}
Let $ M_Q $ be the matrix of $ Q $. By \Cref{corDetMQ}, $ 6^2 \mid \det(M_Q) $ and by \Cref{propDetBinary},
\[
    \frac{\det(M_Q)}{6^2} \equiv 0, 1 \pmod{4}.
\]
Moreover, by \Cref{propRank}, $ \rank(M_Q \bmod{2}) \leq 2 $ and $ \rank(M_Q \bmod{3}) \leq 2 $.

All universal classical quaternary quadratic forms with determinant divisible by $ 36 $ were already listed in \Cref{tabD3}. The only form satisfying the additional conditions is $ 36:\ 2\ 3\ 6\ 0\ 0\ 0 $.

If we let $ G = [3,2\sqrt{6},3,-5, \sqrt{6},-1] $, then
\[
    M_Q = \begin{pmatrix}
        1&0&0&6\\
        0&66&18&0\\
        0&18&5&0\\
        6&0&0&42
    \end{pmatrix} \sim \begin{pmatrix}
        1&0&0&0\\
        0&2&0&0\\
        0&0&3&0\\
        0&0&0&6
    \end{pmatrix}.
\]
(We showed this in \Cref{lemmaExD6}.)
\end{proof}

\begin{proposition}
\label{propExamplesD7}
Let $ K = \Q(\sqrt{7}) $. If $ G $ is a $ \Z $-universal binary generalized form over $ K $ and $ Q $ is the quadratic form associated to $ G $, then $ Q $ is equivalent to the form
\[
    \begin{pmatrix}
        1&0&0&0\\
        0&2&1&0\\
        0&1&4&0\\
        0&0&0&7
    \end{pmatrix}.
\]
Conversely, there exists an equivalent form $ Q $ associated to a $ \Z $-universal binary generalized form over $ K $.
\end{proposition}
\begin{proof}
Let $ M_Q $ be the matrix of $ Q $. By \Cref{corDetMQ}, $ 7^2 \mid \det(M_Q) $ and by \Cref{propDetBinary},
\[
    \frac{\det(M_Q)}{7^2} \equiv 0, 1 \pmod{4}.
\]

The only forms in \cite[Table 5]{Bh} satisfying these conditions are those with determinant equal to $ 49 $, namely
\begin{align*}
    L_1& = 49:\ 2\ 3\ 9\ 2\ 2\ 0,\\
    L_2& = 49:\ 2\ 4\ 7\ 0\ 0\ 2,\\
    L_3& = 49:\ 2\ 5\ 6\ 0\ 2\ 2.
\end{align*}

By \Cref{propRank}, $ \rank(M_Q \bmod{7}) \leq 2 $. We have $ \rank(1\oplus L_1 \bmod{7}) = 3 $, $ \rank(1\oplus L_2 \bmod{7}) = 2 $, and $ \rank(1 \oplus L_3 \bmod{7}) = 3 $. Thus, the only candidate is $ 49:\ 2\ 4\ 7\ 0\ 0\ 2 $.

Let $ G = [3,1+2\sqrt{7},4,-5,-1+2\sqrt{7},3] $. The matrix $ M_Q $ of the quadratic form $ Q $ associated to $ G $ equals
\[
    M_Q = \begin{pmatrix}
        1&0&0&0\\
        0&77&28&14\\
        0&28&11&0\\
        0&14&0&35
    \end{pmatrix} \sim \begin{pmatrix}
        1&0&0&0\\
        0&2&1&0\\
        0&1&4&0\\
        0&0&0&7
    \end{pmatrix}.
\]
(We showed this in \Cref{lemmaExD7}.)
\end{proof}

\begin{proposition}
\label{propExamplesD10}
Let $ K = \Q(\sqrt{10}) $. If $ G $ is a $ \Z $-universal binary generalized form over $ K $ and $ Q $ is the quadratic form associated to $ G $, then $ Q $ is equivalent to the form
\[
    \begin{pmatrix}
        1&0&0&0\\
        0&2&0&0\\
        0&0&5&0\\
        0&0&0&10
    \end{pmatrix}.
\]
Conversely, there exists an equivalent form $ Q $ associated to a $ \Z $-universal binary generalized form over $ K $.
\end{proposition}
\begin{proof}
Let $ M_Q $ be the matrix of $ Q $. By \Cref{corDetMQ}, $ 10^2 \mid \det(M_Q) $. The quadratic forms from \cite[Table 5]{Bh} with determinant divisible by $ 100 $ were already listed in \Cref{tabD5}. By \Cref{propRank}, $ \rank(M_Q \bmod{5}) \leq 2 $. The only form satisfying these criteria is $ 100:\ 2\ 5\ 10\ 0\ 0\ 0 $.

If we let $ G = [3,\sqrt{10},5,-5,\sqrt{10},7] $, then
\[
    M_Q = \begin{pmatrix}
        1&0&0&0\\
        0&110&20&0\\
        0&20&17&20\\
        0&0&20&30
    \end{pmatrix} \sim \begin{pmatrix}
        1&0&0&0\\
        0&2&0&0\\
        0&0&5&0\\
        0&0&0&10
    \end{pmatrix}.
\]
(We showed this in \Cref{lemmaExD10}.)
\end{proof}

\end{document}